\documentclass[12pt]{amsart}
\usepackage{a4wide,enumerate,color}
\allowdisplaybreaks

\let\pa\partial  
\let\na\nabla  
\let\eps\varepsilon  
\newcommand{\N}{{\mathbb N}}  
\newcommand{\R}{{\mathbb R}} 
\newcommand{\diver}{\operatorname{div}}  
\newcommand{\dom}{{\mathcal O}}
\newcommand{\T}{{\mathcal L}}
\newcommand{\F}{{\mathcal F}}
\newcommand{\B}{{\mathcal B}}
\newcommand{\Law}{{\mathcal L}}
\newcommand{\Prob}{\mathbb{P}}
\newcommand{\E}{\mathbb{E}}
\newcommand{\Fil}{\mathbb{F}}

\definecolor{darkblue}{rgb}{0.1,0.1,0.9}

\newtheorem{theorem}{Theorem}   
\newtheorem{lemma}[theorem]{Lemma}   
\newtheorem{proposition}[theorem]{Proposition}   
\newtheorem{remark}[theorem]{Remark}   
\newtheorem{corollary}[theorem]{Corollary}  
\newtheorem{definition}{Definition}

\begin{document}  

\title[Stochastic population cross-diffusion system]{Global martingale solutions 
for a stochastic population cross-diffusion system}

\author[G. Dhariwal]{Gaurav Dhariwal}
\address{Institute for Analysis and Scientific Computing, Vienna University of  
	Technology, Wiedner Hauptstra\ss e 8--10, 1040 Wien, Austria}
\email{gaurav.dhariwal@tuwien.ac.at} 

\author[A. J\"ungel]{Ansgar J\"ungel}
\address{Institute for Analysis and Scientific Computing, Vienna University of  
	Technology, Wiedner Hauptstra\ss e 8--10, 1040 Wien, Austria}
\email{juengel@tuwien.ac.at} 

\author[N. Zamponi]{Nicola Zamponi}
\address{Institute for Analysis and Scientific Computing, Vienna University of  
	Technology, Wiedner Hauptstra\ss e 8--10, 1040 Wien, Austria}
\email{nicola.zamponi@tuwien.ac.at} 

\date{\today}

\thanks{The authors acknowledge partial support from   
the Austrian Science Fund (FWF), grants F65, I3401, P27352, P30000, and W1245} 

\begin{abstract}
The existence of global nonnegative martingale solutions to a 
stochastic cross-diffusion system for an arbitrary but finite
number of interacting population species is shown.
The random influence of the environment is modeled by a multiplicative noise term.
The diffusion matrix is generally neither symmetric nor positive definite,
but it possesses a quadratic entropy structure. This structure allows us to 
work in a Hilbert space framework and to apply a stochastic Galerkin method.
The existence proof is based on energy-type estimates, the tightness criterion of
Brze\'zniak and co-workers, and Jakubowski's generalization of the 
Skorokhod theorem. The nonnegativity is proved by an extension of Stampacchia's
truncation method due to Chekroun, Park, and Temam.
\end{abstract}

\keywords{Shigesada-Kawasaki-Teramoto model, population dynamics, martingale solutions,
tightness, Skorokhod-Jakubowski Theorem, stochastic maximum principle.}  
 
\subjclass[2000]{60H15, 35R60, 60J10, 92D25}  

\maketitle


\section{Introduction}

The dynamics of interacting population species can be described macroscopically
by cross-diffusion equations. A well-known model example is the deterministic
Shigesada-Kawasaki-Teramoto population system \cite{SKT79}. It can be derived
formally from a random-walk model on lattices for transition rates
which depend linearly on the population densities \cite[Appendix A]{ZaJu17}. 
Generalized population cross-diffusion models are obtained when the dependence
of the transition rates on the densities is nonlinear. 
The existence of global weak solutions to these deterministic models was proved 
for an arbitrary number of species in \cite{CDJ18}. 
In this paper, we allow for a random influence of the environment 
and prove the existence of global nonnegative martingale solutions 
to the corresponding stochastic cross-diffusion system. 

More precisely, we consider the cross-diffusion equations 
\begin{equation}\label{1.eq}
  du_i - \diver\bigg(\sum_{j=1}^n A_{ij}(u)\na u_j\bigg)dt 
	= \sum_{j=1}^n\sigma_{ij}(u)dW_j(t)
	\quad\mbox{in }\dom,\ t>0,\ i=1,\ldots,n,
\end{equation}
with no-flux boundary and initial conditions
\begin{equation}\label{1.bic}
  \sum_{j=1}^n A_{ij}(u)\na u_j\cdot\nu = 0\quad\mbox{on }\pa\dom,\ t>0, \quad
	u_i(0)=u_i^0\quad\mbox{in }\dom,\ i=1,\ldots,n,
\end{equation}
where $\dom\subset\R^d$ with $d=2,3$ is a bounded domain with Lipschitz 
boundary, $\nu$ is the exterior unit normal vector to $\pa\dom$, and $u_i^0$ is
a possibly random initial datum.
The solution $u=(u_1,\ldots,u_n):\dom\times[0,T]\times\Omega\to\R^n$
models the density of the $i^{\text{th}}$ population species, where $x\in\dom$
represents the spatial variable, $t\in(0,T)$ the time, and $\omega\in\Omega$
the stochastic variable. The matrix $A(u)=(A_{ij}(u))$ is the diffusion matrix,
$\sigma_{ij}(u)$ is a multiplicative noise term, 
and $W=(W_1,\ldots,W_n)$ is an $n$-dimensional cylindrical Wiener process.
Details on the stochastic framework will be given in section \ref{sec.frame}.

The diffusion coefficients are given by
\begin{equation}\label{1.A}
  A_{ij}(u) = \delta_{ij}\bigg(a_{i0}+\sum_{k=1}^n a_{ik}u_k^2\bigg) + 2a_{ij}u_iu_j, 
	\quad i,j=1,\ldots,n,
\end{equation}
where $a_{i0}>0$ and $a_{ij}>0$. 
This model is derived from an on-lattice model with transition rates
$p_i(u)$, which depend quadratically on the densities, i.e.\
$p_i(u)=a_{i0}+\sum_{k=1}^n a_{ik}u_k^2$ for $i=1,\ldots,n$ \cite{ZaJu17}.
This quadratic structure is essential for our analysis. To understand this, we need to
explain the entropy structure of equations \eqref{1.eq}.

\subsection{Entropy structure}

Generally, the diffusion matrix in \eqref{1.eq}, originating from general transition 
rates in the lattice model, is neither symmetric nor
positive definite which significantly complicates the analysis. However, the
equations possess a formal gradient-flow or entropy structure under certain conditions.
For the sake of simplicity, we sketch this structure in the deterministic context
only and refer to \cite[Chapter~4]{Jue16} for details.
By entropy structure, we mean that there exists a so-called entropy density 
$h:\R_+^n\to\R$ such that, still in the deterministic context, system \eqref{1.eq}
in the entropy variables $w_i:=\pa h/\pa u_i$, $i=1,\ldots,n$,
has a positive semi-definite diffusion matrix $B=(B_{ij})$,
\begin{equation}\label{1.B}
  \pa_t u_i(w) - \diver\bigg(\sum_{j=1}^n B_{ij}\na w_j\bigg) = 0,
\end{equation}
where $B=A(u)h''(u)^{-1}$ is the product of $A(u)$
and the inverse of the Hessian of $h(u)$, and $u(w)=(h')^{-1}(w)$ is the
back transformation. When the transition rates are given by
$p_i(u)=a_{i0}+\sum_{k=1}^n a_{ik}u_k^s$ for some $s\ge 1$, the entropy
density can be chosen as $h(u)=\sum_{i=1}^n \pi_i h_s(u_i)ds$, where
$\pi_i>0$ are some numbers and
$$
  h_s(z) = \left\{\begin{array}{ll}
	z(\log z-1)+1 &\quad\mbox{for } s=1, \\
	z^s/s &\quad\mbox{for }s\neq 1.
	\end{array}\right.
$$
It was shown in \cite{ChJu04} that $B=(B_{ij})$ in \eqref{1.B} is positive 
semi-definite in the two-species case $n=2$ with $\pi_1=\pi_2=1$. This property 
generally does not hold for the $n$-species system. It turns out \cite{CDJ18} 
that $B$ is symmetric, positive semi-definite if the numbers $\pi_i$ are chosen 
such that
$$
  \pi_i a_{ij} = \pi_j a_{ji} \quad\mbox{for all }i,j=1,\ldots,n.
$$
This condition is recognized as the detailed-balance condition
for the Markov chain associated to $(a_{ij})$ and $(\pi_1,\ldots,\pi_n)$ is the
reversible measure. The detailed-balance condition is sufficient but 
not necessary for the positive semi-definiteness of $B$; 
in fact, when self-diffusion dominates
cross-diffusion (see \eqref{1.eta2} for the precise statement) then $B$ is still
positive semi-definite.

The entropy structure also yields a priori estimates. Indeed, let
$H(u)=\int_\dom h(u)dx$ be the so-called entropy.
A computation shows that, still in the absence of the stochastic term,
$$
  \frac{dH}{dt} + \int_\dom\sum_{i,j=1}^n 
	\frac{\pa^2 h}{\pa u_i\pa u_j}(u)A_{ij}(u)\na u_i\cdot\na u_j dx = 0.
$$
Since $B=A(u)h''(u)^{-1}$ is positive semi-definite, this holds true for
$h''(u)A(u)$. Thus, taking into account the special structure of $A(u)$, this
yields gradient estimates (see Lemma \ref{lem.pd} below). 

The gradient-flow structure is the key of the analysis of the
deterministic analog to \eqref{1.eq}, but there are severe difficulties in the
stochastic context. Indeed, neither semigroup techniques \cite{DaZa14,Kru14} 
nor monotonicity arguments \cite{LiRo13} can be applied because of the
properties of the differential operator in \eqref{1.eq}. Stochastic Galerkin methods 
usually work in Hilbert spaces, and generally they cannot
be used since the transformation to entropy variables is nonlinear. 
In order to overcome
these difficulties, we consider quadratic transition rates with $s=2$ which 
makes the transformation to entropy variable linear, 
$$
  w_i = \frac{\pa h}{\pa u_i} = \pi_i h_2'(u_i) = \pi_i u_i.
$$
Still, the diffusion matrix $A(u)$ is not positive definite, but the new diffusion
matrix $B=A(u)\operatorname{diag}(1/\pi_1,\ldots,1/\pi_n)$ is positive semi-definite;
see Lemma \ref{lem.pd}.
This allows us to combine entropy methods for diffusive equations and
stochastic techniques.

\subsection{State of the art}

Before stating our main existence result, let us review the literature.
Fundamental results on stochastic partial differential equations of monotone type
were obtained already in the 1970s by Pardoux \cite{Par76}. More recently,
abstract stochastic evolution equations with locally monotone nonlinearities
\cite{LiRo13} or maximal monotone operators \cite{BaRo18} were
analyzed. The existence of (mild or pathwise strong) solutions to 
quasilinear stochastic evolution equations was proved in, e.g., \cite{DeSt04,HoZh17}.
For these solutions, the driving noise is given in advance. A weaker concept is
given by martingale solutions, where the stochastic basis is unknown a priori
and is given as part of the solution. Existence proofs of such
solutions to nonlinear stochastic evolution equations can be found in 
\cite{BrGa99,ChGo95}. 

Stochastic reaction-diffusion equations are a special class of evolution equations, and
they are investigated in many papers starting from the 1980s \cite{Fla91,FoMi86}. 
There are less results on {\em systems} of stochastic reaction-diffusion equations. 
In \cite{Cer03}, the existence and uniqueness of mild solutions with Lipschitz
continuous multiplicative noise was shown. The result was generalized in \cite{Kun15}
to H\"older continuous multiplicative noise. The existence of maximal pathwise
solutions to stochastic reaction-diffusion systems with polynomial reaction terms
was proved in \cite{NgPh16}. 
More general quasilinear systems were investigated recently in
\cite{KuNe18}, proving the existence of local pathwise mild solutions,
including the Shigesada-Kawasaki-Teramoto cross-diffusion system. The local-in-time
results are not surprising since even in the deterministic case, certain
reaction terms may lead to finite-time blow-up of solutions.
The work \cite{MaMa02} also analyzes population systems
and provides the existence of pathwise unique solutions, but only for two species 
and for Lipschitz continuous nonlinearities. 

Up to our knowledge, 
the population model \eqref{1.eq} with coefficients \eqref{1.A} was not studied
in the literature. In this paper, we prove the existence of global martingale
solutions using the techniques of \cite{BrMo14,BrOn10}. 
We show that the solutions are nonnegative under a natural
condition on the operators $\sigma_{ij}(u)$ using the stochastic maximum principle
of \cite{CPT16}. Since even the uniqueness of weak solutions to the deterministic
analog of \eqref{1.eq}-\eqref{1.A} is not known (see the partial result in
\cite{JuZa16}), we cannot expect to obtain pathwise unique strong solutions.

\subsection{Stochastic framework and main results}\label{sec.frame}

Let $(\Omega,\F,\Prob)$ be a probability space endowed with a complete
right continuous filtration $\Fil=(\F_t)_{t\ge 0}$ and let $H$ be a Hilbert space.
The space $L^2(\dom)$ is the vector space of all square integrable functions
$u:\dom\to\R$ with the inner product $(\cdot,\cdot)_{L^2(\dom)}$.
We fix a Hilbert basis $(e_k)_{k\in\N}$ of $L^2(\dom)$.
The space $L^2(\Omega;H)$ consists of all $H$-valued random variables $u$ with
$$
  \E\|u\|_H^2 := \int_\Omega\|u(\omega)\|_H^2 \Prob(d\omega) < \infty.
$$
Furthermore, the space $H^1(\dom)$ contains all functions $u\in L^2(\dom)$
such that the distributional derivatives $\pa u/\pa x_1,\ldots,\pa u/\pa x_d$ 
belong to $L^2(\dom)$.
Let $Y$ be any separable Hilbert space with orthonormal basis
$(\eta_k)_{k\in\N}$. We denote by 
$$
  \T_2(Y;L^2(\dom)) = \bigg\{L:Y\to L^2(\dom)\mbox{ linear continuous: }
  \sum_{k=1}^\infty\|L\eta_k\|_{L^2(\dom)}^2<\infty\bigg\}
$$
the space of Hilbert-Schmidt operators from $Y$ to $L^2(\dom)$ 
endowed with the norm
$$
  \|L\|_{\T_2(Y;L^2(\dom))}^2 := \sum_{k=1}^\infty\|L\eta_k\|_{L^2(\dom)}^2.
$$

Let $(\beta_{jk})_{j=1,\cdots,n,\, k\in\N}$ be a sequence of independent one-dimensional
Brownian motions and for $j=1,\ldots,n$, let
$W_j(x,t,\omega)=\sum_{k\in\N}\eta_k(x)\beta_{jk}(t,\omega)$ 
be a cylindrical Brownian motion. If $Y_0\supset Y$ is a second auxiliary Hilbert space
such that the map $Y\ni u\mapsto u\in Y_0$ is Hilbert-Schmidt, the series
$W_j=\sum_{k\in\N}\eta_k\beta_{jk}$ converges in $\T_2(\Omega;Y_0)$.

The multiplicative noise terms 
$\sigma := \sigma_{ij}(u,t,\omega):L^2(\dom)\times[0,T]\times\Omega 
\to \T_2(Y; L^2(\dom))$
are assumed to be $\B(L^2(\dom)\otimes[0,T]\otimes\F;\B(\T_2(Y;L^2(\dom))))$-measurable 
and $\Fil$-adapted with the property that 
there exists one constant $C_\sigma>0$ such that for all $u$, $v\in L^2(\dom)$
and $i,j=1,\ldots,n$,
\begin{equation}\label{1.sigma}
\begin{aligned}
  \|\sigma_{ij}(u)\|_{\T_2(Y;L^2(\dom))}^2 
	&\le C_\sigma\big(1+\|u\|^2_{L^2(\dom)}\big), \\
  \|\sigma_{ij}(u)-\sigma_{ij}(v)\|_{\T_2(Y;L^2(\dom))}^2 
	&\le C_\sigma\|u-v\|^2_{L^2(\dom)}.
\end{aligned}
\end{equation}
Here, the $L^2(\dom)$ norm of the function $u=(u_1,\ldots,u_n)$ is understood as
$\|u\|_{L^2(\dom)}^2=\sum_{i=1}^n\|u_i\|_{L^2(\dom)}^2$, and we use this notation
also for other vector-valued or tensor-valued functions.
The expression $\sigma_{ij}(u)dW_j(t)$ formally means that
\begin{equation}\label{1.sigmadW}
  \sigma_{ij}(u)dW_j(t) 
	= \sum_{k,\ell\in\N}\sigma_{ij}^{k\ell}(u)e_\ell d\beta_{jk}(t), \quad\mbox{where }
	\sigma_{ij}^{k\ell}(u) := \big(\sigma_{ij}(u)\eta_k,e_\ell\big)_{L^2(\dom)}.
\end{equation}

Next, we define our concept of solution.

\begin{definition}
Let $T>0$ be arbitrary.
We say that the system $(\widetilde U,\widetilde W,\widetilde u)$ is a {\em global
martingale solution} to \eqref{1.eq}-\eqref{1.A}
if $\widetilde U=(\widetilde\Omega,\widetilde\F,\widetilde\Prob,\widetilde\Fil)$
is a stochastic basis with filtration $\widetilde\Fil=(\widetilde\F_t)_{t\in[0,T]}$,
$\widetilde W$ is a cylindrical Wiener process, and 
$\widetilde u(t)=(\widetilde u_1(t),\ldots,\widetilde u_n(t))$ 
is an $\widetilde\F_t$-adapted stochastic process for all $t\in[0,T]$ such that
for all $i=1,\ldots,n$,
$$
  \widetilde u_i\in L^2(\widetilde\Omega;C^0([0,T];L^2_w(\dom)))\cap
	L^2(\widetilde\Omega;L^2(0,T;H^1(\dom))),
$$
the law of $\widetilde u_i(0)$ is the same as for $u^0_i$, 
and $\widetilde u$ satisfies for all $\phi\in H^1(\dom)$ and all $i=1,\ldots,n$,
\begin{align*}
  (\widetilde u_i(t),\phi)_{L^2(\dom)}
	&= (\widetilde u_{i}(0),\phi)_{L^2(\dom)} 
	- \sum_{j=1}^n\int_0^t\big\langle\diver\big(A_{ij}(\widetilde u(s))
	\na\widetilde u_j(s)\big),\phi\big\rangle ds \\
  &\phantom{xx}{}+ \bigg(\sum_{j=1}^n\int_0^t\sigma_{ij}(\widetilde u(s))
	d\widetilde W_j(s),\phi\bigg)_{L^2(\dom)}.
\end{align*}
\end{definition}

The brackets $\langle\cdot,\cdot\rangle$ signify the duality pairing between 
$H^1(\dom)'$ and $H^1(\dom)$, i.e.
$$
  \big\langle\diver\big(A_{ij}(\widetilde u)
	\na\widetilde u_j\big),\phi\big\rangle
	= -\int_\dom A_{ij}(\widetilde u)	\na\widetilde u_j\cdot\na\phi dx.
$$
As mentioned before, the new diffusion matrix $B$ in \eqref{1.B} is
positive definite only under an additional assumption, namely either
\begin{align}
  & \pi_ia_{ij} = \pi_ja_{ji}\mbox{ for }i\neq j\quad\mbox{and}\quad
	\alpha_1 := \min_{i=1,\ldots,n}\bigg(a_{ii}-\frac13\sum_{j=1,\,j\neq i}^n a_{ij}
	\bigg) > 0, \quad\mbox{or} \label{1.eta1} \\
  & \alpha_2 := \min_{i=1,\ldots,n}\bigg(a_{ii} - \frac13\sum_{j=1,\,j\neq i}^n
	\big((a_{ij}+a_{ji}) - 2\sqrt{a_{ij}a_{ji}}\big)\bigg) > 0. \label{1.eta2}
\end{align}
Our main result is as follows.

\begin{theorem}[Existence of global martingale solution]\label{thm.ex}
Let $T>0$ be arbitrary, $d\le 3$, and $u_0\in L^2(\dom)$.
Let $\sigma=(\sigma_{ij})_{i,j=1}^n$ 
with $\sigma_{ij}:L^2(\dom)\times[0,T]\times\Omega\to\T_2(Y;L^2(\dom))$
satisfy \eqref{1.sigma}, $a_{i0}>0$, $a_{ij}>0$ for $i,j=1,\ldots,n$, 
and let either \eqref{1.eta1} or \eqref{1.eta2} hold. 
Then there exists a global martingale solution to \eqref{1.eq}-\eqref{1.A}.
If additionally, $u_i^0\ge 0$ a.e.\ in $\dom$, $\Prob$-a.s.\ for $i=1,\ldots,n$ and
\begin{equation}\label{1.sigma2}
  \sum_{j=1}^n\|\sigma_{ij}(u)\|_{\T_2(Y;L^2(\dom))} \le C\|u_i\|_{L^2(\dom)},
\end{equation}
then the population densities are nonnegative $\Prob$-a.s.
\end{theorem}

\begin{remark}[Discussion of the assumptions]\label{rem.assump}\rm
(i) We can also choose random initial data, see Remark \ref{rem.ic}.
We need additionally that $\E\|u^0\|_{L^2(\dom)}^p<\infty$ for $p=24/(4-d)$.
This condition is needed to derive a higher-order estimate for $u_i$.
It can be weakened to smaller values of $p$ by refining the Gagliardo-Nirenberg 
argument in the proof of Lemma \ref{lem.3}.

(ii) Assumption \eqref{1.sigma} on $\sigma_{ij}$ seems to be quite natural.
In \cite{Kun15}, the multiplicative noise was assumed to be only 
H\"older continuous, but the matrix $(\sigma_{ij}(u))$ is needed to be diagonal,
which we do not assume. Condition \eqref{1.sigma2} implies that 
$\sum_{j=1}^n\sigma_{ij}(u)=0$ if $u_i=0$, which is a natural condition to obtain
the nonnegativity of $u_i$.

(iii) The existence of solutions to the deterministic version of 
\eqref{1.eq}-\eqref{1.A} can be shown also for
vanishing coefficients $a_{i0}=0$ \cite{CDJ18}. This seems to be not
possible in the stochastic framework, since the condition $a_{i0}>0$ is needed
to derive estimates for $\na u_i$ in $L^2(\dom)$ $\Prob$-a.s., and these
estimates are necessary to work in the Hilbert space $H^1(\dom)$.

(iv) Conditions \eqref{1.eta1} and \eqref{1.eta2} on the matrix coefficients are
probably not optimal. For local-in-time existence of solutions to the determinstic
analog of \eqref{1.eq}, only the positivity of the real parts of the eigenvalues
of $A(u)$ is needed \cite{Ama90}. This condition is generally not sufficient
to ensure global solvability. A sufficient condition for the global existence
for general quasilinear
evolution equations is provided by uniform $W^{1,p}(\dom)$ bounds with $p>d$ 
\cite[Theorem 15.3]{Ama93}, but it is difficult to prove this regularity for
solutions to cross-diffusion systems. Conditions \eqref{1.eta1} 
and \eqref{1.eta2} are currently the best available assumptions to guarantee 
the existence of global solutions, even in the deterministic framework. 
\qed
\end{remark}

\subsection{Ideas of the proof of Theorem \ref{thm.ex}}

We sketch the main steps of the proof. The full proof is given in section \ref{sec.ex}.
First, we show the existence of a pathwise unique strong solution $u^{(N)}$ to a
stochastic Galerkin approximation of \eqref{1.eq}-\eqref{1.A}, 
where $N\in\N$ is the Galerkin dimension. 
Estimates uniform in $N$ are derived from a stochastic version of the entropy inequality
(which is made rigorous using It\^o's formula in section \ref{sec.unif})
\begin{align*}
  &\E H(u^{(N)}(t)) - \E H(u^{(N)}(0))
	+ \sum_{i,j=1}^n\E\int_0^t\int_\dom \pi_iA_{ij}(u^{(N)})\na u_i^{(N)}\cdot
	\na u_j^{(N)} dxds \\
	&\le \frac12\E\int_0^t\|P^{1/2}\Pi_N\sigma(u^{(N)})\|_{\T_2(Y;L^2(\dom))}^2 ds
	+ \sum_{i,j=1}^n\E\int_0^t\int_\dom\pi_i u^{(N)}_i\sigma_{ij}(u^{(N)})dW_j(s) dx,
\end{align*}
where $\Pi_N$ is the projection on the finite-dimensional Galerkin space, 
$$
  H(u) = \sum_{i=1}^n\int_\dom \pi_i h_2(u_i)dx 
	= \sum_{i=1}^n\frac{\pi_i}{2}\int_\dom u_i^2 dx = \frac12\|P^{1/2}u\|_{L^2(\dom)}^2
$$
is the quadratic entropy, and $P=\operatorname{diag}(\pi_1,\ldots,\pi_n)$,
$P^{1/2}=\operatorname{diag}(\pi_1^{1/2},\ldots,\pi_n^{1/2})$.
Since $PA(u^{(N)})$ is positive definite, the last term on the left-hand side
yields uniform gradient estimates. The first integral on the right-hand side
is bounded from above by the entropy $H$ (up to some additive constant), 
using assumption \eqref{1.sigma}, and the second integral is estimated using
the Burkholder-Davis-Gundy inequality (see Proposition \ref{prop.bdg} in the appendix).

Next, the tightness of the laws $\Law(u^{(N)})$ in the topological space $Z_T$,
defined in \eqref{3.Z} below,
is proved by applying a criterion of Brze\'zniak, Goldys, and Jegaraj \cite{BGJ13}.
Because of the low regularity properties of the solutions, $Z_T$
cannot be chosen to be a metric space and we cannot apply the Skorokhod representation
theorem, as usually done in the literature (e.g.\ \cite{DHV16,NgPh16}).  
This problem is overcome by using 
Jakubowski's generalization of the Skorokhod theorem,
which holds for topological spaces with a separating-points property
(Theorem \ref{thm.skoro}).
Then there exists a subsequence of $(u^{(N)})$ (not relabeled), another probability space,
and random variables $(\widetilde u^{(N)},\widetilde W^{(N)})$ having the same law
as $(u^{(N)},W)$ and $(\widetilde u^{(N)},\widetilde W^{(N)})$ converges
to $(\widetilde u,\widetilde W)$ in the topology of $Z_T$. 
Because of the gradient estimates,
we conclude in particular the strong convergence $\widetilde u^{(N)}\to
\widetilde u$ in $L^2(\dom\times(0,T))$ $\Prob$-a.s. This, together with
further convergences resulting from the relative compactness in $Z_T$,
allows us to pass to the limit $N\to\infty$ in the Galerkin approximation,
showing that $(\widetilde u,\widetilde W)$ is a global martingale solution
to \eqref{1.eq}. 

From the application viewpoint, 
we expect that the population densities $u_i(t)$ are nonnegative 
$\Prob$-a.s.\ if this holds initially. The problem is that
generally, maximum principle arguments cannot be applied to
cross-diffusion systems. System \eqref{1.eq}, \eqref{1.A}, however, possesses a special
structure. Indeed, we may write \eqref{1.eq} as
$$
  du_i - \diver\bigg(\bigg(a_{i0}+\sum_{k=1}^n a_{ik}u_k^2\bigg)\na u_i 
	+ u_iF_i[u]\bigg) = \sum_{j=1}^n\sigma_{ij}(u)dW_j(t),
$$
and $F_i$ depends on $u_j$ and $\na u_j$ for $j\neq i$. The term $u_iF_i[u]$
can be interpreted as a drift term which vanishes if $u_i=0$. If we assume that
$\sigma_{ij}(u)=0$ if $u_i=0$ then a maximum principle can be applied. 

More precisely, we employ the stochastic Stampacchia-type maximum principle 
due to Chekroun, Park, and Temam \cite{CPT16}. 
The idea is to regularize the test function 
$(\widetilde u_i^{(N)})^-=\max\{0,-\widetilde u_i^{(N)}\}$
by some smooth function $F_\eps(\widetilde u_i^{(N)})$, to apply the It\^o formula for 
$\E\int F_\eps(\widetilde u_i^{(N)})dx$, and then to pass to the limits $N\to\infty$ and
$\eps\to 0$ leading to the inequality
$$
  \E\|\widetilde u_i(t)^-\|_{L^2(\dom)}^2 
	\le \E\int_0^t\|\widetilde u_i(s)^-\|_{L^2(\dom)}^2 ds.
$$
Gronwall's lemma show that $\widetilde u_i(t)^-=0$ a.e.\ in $\dom$,
which proves the nonnegativity of $\widetilde u_i$ $\Prob$-a.s.

In order to make the manuscript accessible also to non-experts of stochastic
partial differential equations, we recall some known results from stochastic analysis used
in this paper in Appendix \ref{app}. As the tightness criterion of \cite{BGJ13}
is probably less known, we present the details directly in the proof of Theorem
\ref{thm.ex} in section \ref{sec.tight}.


\section{Proof of the existence theorem}\label{sec.ex}

\subsection{An algebraic property}\label{sec.alg}

We recall the following result on the positive definite\-ness of the
new diffusion matrix, taken from \cite[Lemma 3]{CDJ18} by choosing $s=2$. 

\begin{lemma}\label{lem.pd}
Let $\pi_1,\ldots,\pi_n>0$
and $P=\operatorname{diag}(\pi_1,\ldots,\pi_n)\in\R^{n\times n}$. 
Let either condition \eqref{1.eta1} or \eqref{1.eta2} hold.
Then $PA(u)$ is positive definite, i.e., 
it holds for any $z=(z_1,\ldots,z_n)\in\R^n$ and $u=(u_1,\ldots,u_n)\in\R^n$,
$$
  \sum_{i,j=1}^n \pi_iA_{ij}(u)z_iz_j
	\ge \sum_{i=1}^n\pi_i a_{i0}z_i^2 + 3\alpha\sum_{i=1}^n \pi_i u_i^2z_i^2,
$$
where $\alpha=\alpha_1$ if \eqref{1.eta1} holds and 
$\alpha=\alpha_2$ if \eqref{1.eta2} is
satisfied. In the latter case, we may choose $\pi_i=1$ for all $i=1,\ldots,n$.
\end{lemma}


\subsection{Stochastic Galerkin approximation}\label{sec.gal}

We fix an orthonormal basis $(e_k)_{k\ge 1}$ of $L^2(\dom)$ and a number $N\in\N$
and set $H_N=\operatorname{span}\{e_1,\ldots,e_N\}$. We introduce the
projection operator $\Pi_N:L^2(\dom)\to H_N$,
$$
  \Pi_N(v) = \sum_{i=1}^N(v,e_i)_{L^2(\dom)}e_i, \quad v\in L^2(\dom).
$$
The approximate problem is the following system of stochastic
differential equations,
\begin{align}
  & du_i^{(N)} - \Pi_N\diver\bigg(\sum_{j=1}^n A_{ij}(u^{(N)})\na u_j^{(N)}
	\bigg) dt
	= \Pi_N\bigg(\sum_{j=1}^n\sigma_{ij}(u^{(N)})\bigg)dW_j(t),
	\label{2.approx1} \\
	& u_i^{(N)}(0) = \Pi_N(u_i^0), \quad i=1,\ldots,n. \label{2.approx2}
\end{align}

\begin{lemma}\label{lem.ex}
Let Assumptions \eqref{1.eta1} or \eqref{1.eta2} hold. Then there
exists a pathwise unique strong solution to \eqref{2.approx1}-\eqref{2.approx2}.
\end{lemma}

\begin{proof}
We apply Theorem \ref{thm.sde} in Appendix \ref{app} to
\begin{equation}\label{2.pi}
  \pi\cdot du = a(u)dt + b(u)dW(t), \quad t>0, \quad u(0)=\Pi_N (u^0),
\end{equation}
where
\begin{align*}
  & a=(a_1,\ldots,a_n):H_N\to\R^n, \quad
	a_i(u)=\Pi_N\diver\bigg(\sum_{j=1}^n \pi_iA_{ij}(u)\na u_j\bigg), \\
  & b_{ij}:H_N \to\T_2(Y;H_N), \quad b_{ij}(u) = \pi_i\Pi_N\sigma_{ij}(u),
\end{align*}	
and the numbers $\pi_1,\ldots,\pi_n>0$ are given by \eqref{1.eta1}.
Observe that this problem is equivalent to \eqref{2.approx1} after 
componentwise division by $\pi_i$.
It is sufficient to verify Assumptions
\eqref{2.ab1}-\eqref{2.ab2}. Let $R>0$, $T>0$, and $\omega\in\Omega$ and let
$u$, $v\in H_N$ with $\|u\|_{H_N}$, $\|v\|_{H_N}\le R$. Then, using the
positive definiteness of $PA$, according to Lemma \ref{lem.pd}, 
and the equivalence of norms on $H_N$, 
\begin{align*}
  (a(u)-a(v),u-v)_{H_N} 
	&= -\sum_{i,j=1}^n\int_{\dom}\pi_iA_{ij}(u)\na(u_i-v_i)\cdot\na(u_j-v_j)dx \\
	&\phantom{xx}{}
	+ \sum_{i,j=1}^n\int_{\dom}\pi_i(A_{ij}(u)-A_{ij}(v))\na(u_i-v_i)\cdot\na v_j dx \\
	&\le C\|A(u)-A(v)\|_{L^2(\dom)}\|\na(u-v)\|_{L^2(\dom)}\|\na v\|_{L^\infty(\dom)} \\
	&\le C(N,R)\|u-v\|_{H_N}^2,
\end{align*}
where the constant $C(N,R)>0$ depends on $N$ and $R$.
In the last step we have used the fact that $A_{ij}(u)$ is locally Lipschitz continuous.
Hence, together with assumption \eqref{1.sigma} on $\sigma$, the local weak
monotonicity condition \eqref{2.ab1} holds. To verify the weak coercivity
condition \eqref{2.ab2}, we take $u\in H_N$ with $\|u\|_{H_N}\le R$ and
employ again the positive definiteness of $PA$:
\begin{align*}
  (a(u),u)_{H_N} + \|b(u)\|_{\T_2(Y;H_N)}^2
	&= -\sum_{i,j=1}^n\int_{\dom}\pi_iA_{ij}(u)\na u_i\cdot\na u_j dx
	+ \|P^{1/2}\sigma(u)\|_{\T_2(Y;H_N)}^2 \\
	&\le C_\sigma(1+\|u\|_{H_N}^2),
\end{align*}
where we recall that $P^{1/2}=\operatorname{diag}(\pi_1^{1/2},\ldots,\pi_n^{1/2})$.
Therefore, the lemma follows after applying Theorem \ref{thm.sde}.
\end{proof}


\subsection{Uniform estimates}\label{sec.unif}

We prove some energy-type estimates uniform in $N$.

\begin{lemma}[A priori estimates]\label{lem.est1}
Let $T>0$ and let $u^{(N)}$ be the pathwise unique strong solution 
to \eqref{2.approx1}-\eqref{2.approx2} on $[0,T]$. Then there exists a constant
$C_1>0$ which depends on $\E\|u^0\|_{L^2(\dom)}^2$, $C_\sigma$, and 
$T$ but not on $N$ such that
\begin{align}
  \sup_{N\in\N}\E\bigg(\sup_{t\in(0,T)}\|u^{(N)}\|_{L^2(\dom)}^2\bigg)
	&\le C_1, \label{3.est1} \\
  \sup_{N\in\N}\E\bigg(\int_0^T\|\na u^{(N)}\|_{L^2(\dom)}^2 dt\bigg)
	&\le C_1, \label{3.est2} \\
	\alpha\sup_{N\in\N}\E\bigg(\int_0^T\big\|\na (u^{(N)})^2\big\|_{L^2(\dom)}^2 dt\bigg)
	&\le C_1, \label{3.est3}
\end{align}
and $\alpha=\alpha_1$ if \eqref{1.eta1} holds, 
$\alpha=\alpha_2$ if \eqref{1.eta2} holds.
\end{lemma}

We remark that \eqref{3.est1} shows that $(u^{(N)})$ is bounded in
$L^2(\dom\times(0,T)\times\Omega)$, so together with \eqref{3.est2}, we infer
a uniform bound for $u^{(N)}$ in $L^2((0,T)\times\Omega;H^1(\dom))$.

\begin{proof}
We apply the It\^o formula (Theorem \ref{thm.ito}) to the 
process $X(t)=u^{(N)}(t)$, where $u^{(N)}$ solves \eqref{2.pi}:
\begin{align}
  \frac12\|&P^{1/2}u^{(N)}(t)\|_{L^2(\dom)}^2
	- \frac12\|\Pi_N(P^{1/2}u^0)\|_{L^2(\dom)}^2 \nonumber \\
	&= \sum_{i,j=1}^n\int_0^t\big(u_i^{(N)}(s),
	\Pi_N\diver(\pi_i A_{ij}(u^{(N)}(s))
	\na u_j^{(N)}(s))\big)_{L^2(\dom)}ds \nonumber \\
	&\phantom{xx}{}+ \frac12\int_0^t\big\|\Pi_N(P^{1/2}\sigma(u^{(N)}(s)))
	\big\|_{\T_2(Y;L^2(\dom))}^2ds \nonumber \\
	&\phantom{xx}{}
	+ \sum_{i,j=1}^n\int_0^t\big(u_i^{(N)}(s),\Pi_N(\pi_i\sigma_{ij}(u^{(N)}(s)))dW_j(s)
	\big)_{L^2(\dom)} \nonumber \\
	&= -\sum_{i,j=1}^n\int_0^t\big(\na u_i^{(N)}(s),\pi_iA_{ij}(u^{(N)}(s))
	\na u_j^{(N)}(s)\big)_{L^2(\dom)}ds \nonumber \\
	&\phantom{xx}{}+ \frac12\int_0^t\big\|\Pi_N(P^{1/2}\sigma(u^{(N)}(s)))
	\big\|_{\T_2(Y;L^2(\dom))}^2ds \nonumber \\
	&\phantom{xx}{}+ \sum_{i,j=1}^n\int_0^t\pi_i \big(u_i^{(N)}(s),
	\sigma_{ij}(u^{(N)}(s))dW_j(s)\big)_{L^2(\dom)}. \label{2.aux1}
\end{align}
The first term on the right-hand side can be estimated by using Lemma \ref{lem.pd}:
\begin{align*}
  \sum_{i,j=1}^n&\big(\na u_i^{(N)}(s),\pi_iA_{ij}(u^{(N)}(s))
	\na u_j^{(N)}\big)_{L^2(\dom)} \\
	&\ge \sum_{i=1}^n\pi_i a_{i0}\int_{\dom}|\na u_i^{(N)}|^2dx
	+ 3\alpha\sum_{i=1}^n\pi_i\int_{\dom}|u_i^{(N)}|^2|\na u_i^{(N)}|^2 dx \\
	&\ge C\|\na u^{(N)}\|_{L^2(\dom)}^2 + C\alpha\|\na (u^{(N)})^2\|_{L^2(\dom)}^2,
\end{align*}
where $(u^{(N)})^2=((u_1^{(N)})^2,\ldots,(u_n^{(N)})^2)$ and here and in the
following, $C>0$ is a generic constant independent of $N$ with values changing
from line to line. Therefore, \eqref{2.aux1} becomes
\begin{align}
  \frac12\|&P^{1/2}u^{(N)}(t)\|_{L^2(\dom)}^2
	+ C\int_0^t\|\na u^{(N)}(s)\|_{L^2(\dom)}^2ds
	+ C\alpha\int_0^t\|\na(u^{(N)}(s)^2)\|_{L^2(\dom)}^2ds \nonumber \\
	&\le \frac12\|P^{1/2}u^0\|_{L^2(\dom)}^2
	+ \frac12\int_0^t\big\|P^{1/2}\sigma(u^{(N)}(s))\big\|_{\T_2(Y;L^2(\dom))}^2 ds 
	\label{3.aux2} \\
	&\phantom{xx}{}+ \sum_{i,j=1}^n\int_0^t\pi_i\big(u_i^{(N)}(s),
	\sigma_{ij}(u^{(N)}(s))dW_j(s)\big)_{L^2(\dom)}.
	\nonumber
\end{align}

For the second integral on the right-hand side, we take into account
assumption \eqref{1.sigma}:
\begin{align*}
  \frac12\int_0^t&\big\|P^{1/2}\sigma(u^{(N)}(s))\big\|_{\T_2(Y;L^2(\dom))}^2 ds
	\le C\int_0^t\big\|\sigma(u^{(N)}(s))\big\|_{\T_2(Y;L^2(\dom))}^2 ds \\
	&\le C\int_0^t\big(1 + \|u^{(N)}\|_{L^2(\dom)}^2\big)ds 
	= Ct + C\int_0^t\|u^{(N)}\|_{L^2(\dom)}^2 ds.
\end{align*}
To estimate the last integral in \eqref{3.aux2}, we observe that,
since the process $u^{(N)}$ is $H_N$-valued and a solution to \eqref{2.approx1},
the process
$$
  \mu^{(N)}(t) = \sum_{i,j=1}^n\int_0^t\pi_i\big(u_i^{(N)},
	\sigma_{ij}(u^{(N)}(s))dW_j(s)\big)_{L^2(\dom)}, \quad t\in[0,T],
$$
is an $\F_t$-martingale. Then, by the Burkholder-Davis-Gundy inequality 
(see Proposition \ref{prop.bdg}), we have
\begin{align*}
  \E\bigg(\sup_{t\in(0,T)}&\bigg|\sum_{i,j=1}^n\int_0^t\pi_i\big(u_i^{(N)},
	\sigma_{ij}(u^{(N)}(s))dW_j(s)\big)_{L^2(\dom)}\bigg|\bigg) \\
	&\le C\E\bigg(\int_0^T\|u^{(N)}(s)\|_{L^2(\dom)}^2
	\big\|\sigma(u^{(N)}(s))\big\|_{\T_2(Y;L^2(\dom))}^2\bigg)^{1/2},
\end{align*}
and by the H\"older inequality, assumption \eqref{1.sigma} on $\sigma$, and
the Young inequality, we obtain
\begin{align}
  \E&\sup_{t\in(0,T)}\bigg|\sum_{i,j=1}^n\int_0^t\pi_i\big(u_i^{(N)},
	\sigma_{ij}(u^{(N)}(s))dW_j(s)\big)_{L^2(\dom)}\bigg| \nonumber \\
	&\le C\E\bigg\{\bigg(\sup_{t\in[0,T]}\|u^{(N)}(t)\|_{L^2(\dom)}^2\bigg)^{1/2}
	C_\sigma^{1/2}
	\bigg(\int_0^T\big(1+\|u^{(N)}\|_{L^2(\dom)}^2\big)ds\bigg)^{1/2}\bigg\}
	\nonumber \\
  &\le \frac14\E\bigg(\sup_{t\in[0,T]}\|u^{(N)}(t)\|_{L^2(\dom)}^2\bigg)
	+ C\bigg(T + \E\int_0^T\|u^{(N)}\|_{L^2(\dom)}^2 ds\bigg). \label{3.dW}
\end{align}
We take in \eqref{3.aux2} the supremum over $t\in(0,T)$ and the mathematical
expectation and use the inequality 
$\|P^{1/2}u^{(N)}\|_{L^2(\dom)}\ge C\|u^{(N)}\|_{L^2(\dom)}$ for some constant $C>0$
only depending on $\pi_1,\ldots,\pi_n$
and the previous estimates to conclude that 
\begin{align}
  \frac14\E\bigg(&\sup_{t\in[0,T]}\|u^{(N)}(t)\|_{L^2(\dom)}^2\bigg)
	+ C\E\int_0^t\|\na u^{(N)}(s)\|_{L^2(\dom)}^2 ds
	+ C\alpha\E\int_0^t\|\na(u^{(N)}(s)^2)\|_{L^2(\dom)}^2ds \nonumber \\
	&\le CT + C\E\big(\|u^0\|_{L^2(\dom)}^2\big)
	+ C\int_0^T\E\bigg(\sup_{t\in[0,\tau]}\|u^{(N)}(\tau)\|_{L^2(\dom)}^2\bigg)ds.
	\label{3.aux3}
\end{align}
We infer from the Gronwall lemma that
$$
  \sup_{N\in\N}\E\bigg(\sup_{t\in[0,T]}\|u^{(N)}(t)\|_{L^2(\dom)}^2\bigg)
	\le C,
$$
where $C>0$ depends on $\E\|u^0\|_{L^2(\dom)}^2$, $C_\sigma$, and $T$.
This proves \eqref{3.est1}. Inserting the previous estimate into \eqref{3.aux3},
we deduce immediately estimates \eqref{3.est2} and \eqref{3.est3}.
\end{proof}

We need a higher-order moment estimate,
which is proved in the following lemma. 

\begin{lemma}\label{lem.estp}
Let $T>0$ and let $u^{(N)}$ be the pathwise unique strong solution 
to \eqref{2.approx1}-\eqref{2.approx2} on $[0,T]$. Furthermore, let $p>2$
and $\E\|u^0\|_{L^2(\dom)}^p<\infty$. Then there exists a constant
$C_2>0$ which depends on $p$, $\E\|u^0\|_{L^2(\dom)}^p$, $C_\sigma$, and 
$T$ but not on $N$ such that
\begin{equation}\label{3.estp}
  \sup_{N\in\N}\E\bigg(\sup_{t\in(0,T)}\|u^{(N)}\|_{L^2(\dom)}^p\bigg) \le C_2.
\end{equation}
\end{lemma}

\begin{proof}
We take the supremum over $t\in(0,T)$ in \eqref{3.aux2} and neglect the second and
third terms on the left-hand side. Then, raising both sides to the the power $p/2$
and applying the H\"older inequality, we find that
\begin{align*}
  \sup_{t\in(0,T)}\|u^{(N)}\|_{L^2(\dom)}^p
	&\le C\|u^0\|_{L^2(\dom)}^p 
	+ CT^{p/2-1}\int_0^T\big\|\sigma(u^{(N)}(s))\big\|_{\T_2(Y;L^2(\dom))}^p ds \\
	&\phantom{xx}{}+ C\bigg(\sup_{t\in(0,T)}\sum_{i,j=1}^n\int_0^t\big(u_i^{(N)}(s),
	\pi_i\sigma_{ij}(u^{(N)}(s))dW_j(s)\big)_{L^2(\dom)}\bigg)^{p/2}.
\end{align*}
Taking the mathematical expectation and using assumption \eqref{1.sigma}, it follows that
\begin{align}
  \E\bigg(&\sup_{t\in(0,T)}\|u^{(N)}\|_{L^2(\dom)}^p\bigg)
	\le C + C\E\|u^0\|_{L^2(\dom)}^p
	+ C\E\int_0^T\|u^{(N)}(s)\|_{L^2(\dom)}^{p} ds \nonumber \\
  &\phantom{xx}{}+ C\E\bigg(\sup_{t\in(0,T)}\sum_{i,j=1}^n\int_0^t\big(u_i^{(N)}(s),
	\pi_i\sigma_{ij}(u^{(N)}(s))dW_j(s)\big)_{L^2(\dom)}\bigg)^{p/2}. \label{3.aux4}
\end{align}
For the last term, we use the Burkholder-Davis-Gundy and Young inequalities,
\begin{align*}
  \E\bigg(&\sup_{t\in(0,T)}\sum_{i,j=1}^n\int_0^t\big(u^{(N)}(s),
	\pi_i\sigma_{ij}(u^{(N)}(s))dW_j(s)\big)_{L^2(\dom)}\bigg)^{p/2} \\
	&\le C\E\bigg(\int_0^T\|u^{(N)}(s)\|_{L^2(\dom)}^2
	\big\|\sigma(u^{(N)}(s))\big\|_{\T_2(Y;L^2(\dom))}^2 ds\bigg)^{p/4} \\
  &\le C\E\bigg\{\bigg(\sup_{t\in[0,T]}\|u^{(N)}(t)\|_{L^2(\dom)}^2\bigg)^{p/4}
	C_\sigma^{p/4}\bigg(\int_0^T\big(1+\|u^{(N)}\|_{L^2(\dom)}^2\big)ds
	\bigg)^{p/4}\bigg\} \\
	&\le C\E\bigg\{\bigg(\sup_{t\in[0,T]}\|u^{(N)}(t)\|_{L^2(\dom)}^p\bigg)^{1/2}
	\bigg(\int_0^T\big(1+\|u^{(N)}\|_{L^2(\dom)}^p\big)ds\bigg)^{1/2}\bigg\} \\
	&\le \frac12\E\bigg(\sup_{t\in[0,T]}\|u^{(N)}(t)\|_{L^2(\dom)}^p\bigg)
	+ C\E\int_0^T\big(1+\|u^{(N)}\|_{L^2(\dom)}^p\big)ds.
\end{align*}
Inserting this estimate into \eqref{3.aux4} and observing that the first term
on the right-hand side of the previous inequality 
can be absorbed by the first term on the left-hand side of
\eqref{3.aux4}, we infer that
\begin{align*}
  \E\bigg(\sup_{t\in(0,T)}\|u^{(N)}\|_{L^2(\dom)}^p\bigg)
	&\le C + C\E\|u^0\|_{L^2(\dom)}^p
	+ C\E\int_0^T\sup_{\tau\in(0,s)}\|u^{(N)}(\tau)\|_{L^2(\dom)}^{p} ds \\
	&\phantom{xx}{}+ C\E\int_0^T\big(1+\|u^{(N)}\|_{L^2(\dom)}^p\big)ds.
\end{align*}
Then the Gronwall inequality implies that
$$
  \E\bigg(\sup_{t\in(0,T)}\|u^{(N)}\|_{L^2(\dom)}^p\bigg)	\le C,
$$
which concludes the proof.
\end{proof}

The previous lemma allows us to improve slightly the regularity of $u^{(N)}$.

\begin{lemma}\label{lem.3}
Let $T>0$ and let $u^{(N)}$ be the pathwise unique strong solution 
to \eqref{2.approx1}-\eqref{2.approx2} on $[0,T]$. Then
$(u_i^{(N)})^2\in L^3((0,T)\times\Omega;L^2(\dom))$ for $i=1,\ldots,N$
and, for some constant $C_3>0$,
$$
  \E\int_0^T\|(u^{(N)})^2\|_{L^2(\dom)}^3 dt \le C_3,
$$
where $(u^{(N)})^2$ is the vector with the coefficients $(u_i^{(N)})^2$ for
$i=1,\ldots,N$. 
\end{lemma}

\begin{proof}
By the Gagliardo-Nirenberg inequality with $\theta=d/(2+d)$ and the H\"older
inequality with $q=2(2+d)/(3d)$ and $q'=2(2+d)/(4-d)$ (here, we need that $d\le 3$), 
we find that
\begin{align*}
  \E\int_0^T\|(u^{(N)})^2\|_{L^2(\dom)}^3 dt 
	&\le C\E\int_0^T\|(u^{(N)})^2\|_{H^1(\dom)}^{3d/(2+d)}
	\|(u^{(N)})^2\|_{L^1(\dom)}^{6/(2+d)}dt \\
	&\le C\E\bigg(\sup_{t\in(0,T)}\|u^{(N)}\|_{L^2(\dom)}^{12/(2+d)}
	\int_0^T\|(u^{(N)})^2\|_{H^1(\dom)}^{3d/(2+d)}dt\bigg) \\
	&\le C\bigg\{\E\bigg(\sup_{t\in(0,T)}\|u^{(N)}\|_{L^2(\dom)}^{24/(4-d)}
	\bigg)\bigg\}^{1/q'}
	\bigg\{\E\int_0^T\|(u^{(N)})^2\|_{H^1(\dom)}^2 dt\bigg\}^{1/q}.
\end{align*}
The first factor is uniformly bounded by \eqref{3.estp} with $p=24/(4-d)$
and the second factor is uniformly bounded as a consequence of \eqref{3.est1}
and \eqref{3.est2}.
\end{proof}


\subsection{Tightness}\label{sec.tight}

The aim of this subsection is to prove that the sequence of laws of $u^{(N)}$ 
is tight on a certain topological space. For this, we introduce the following
spaces:
\begin{itemize}
\item $C^0([0,T];H^3(\dom)')$ is the space of continuous functions
$u:[0,T]\to H^3(\dom)'$ with the topology $\mathcal{T}_1$ induced by the norm
$\|u\|_{C^0([0,T];H^3(\dom)')}=\sup_{t\in(0,T)}\|u(t)\|_{H^3(\dom)'}$;
\item $L^2_w(0,T;H^1(\dom))$ is the space $L^2(0,T;H^1(\dom))$ with the weak
topology $\mathcal{T}_2$;
\item $L^2(0,T;L^2(\dom))$ is the space of square integrable functions 
$u:(0,T)\to L^2(\dom)$ with the topology $\mathcal{T}_3$ induced by the norm
$\|\cdot\|_{L^2(0,T;L^2(\dom))}$;
\item $C^0([0,T];L^2_w(\dom))$ is the space of weakly continuous functions
$u:[0,T]\to L^2(\dom)$ endowed with the weakest topology $\mathcal{T}_4$ such that
for all $h\in L^2(\dom)$, the mappings
$$
  C^0([0,T];L_w^2(\dom))\to C^0([0,T];\R), \quad u\mapsto (u(\cdot),h)_{L^2(\dom)},
$$
are continuous.
\end{itemize}

In particular, convergence in $C^0([0,T];L^2_w(\dom))$ means the following:
$u_n\to u$ in $C^0([0,T];$ $L^2_w(\dom))$ as $n\to\infty$ holds if and only if
$$
  \lim_{n\to\infty}\sup_{t\in(0,T)}|(u_n(t)-u(t),h)_{L^2(\dom)}| = 0
	\quad\mbox{for all }h\in L^2(\dom).
$$

We need another space: Let $r>0$ and $B:=\{u\in L^2(\dom):\|u\|_{L^2(\dom)}\le r\}$.
Let $q$ be the metric compatible with the weak topology on $B$. We define the
following subspace of $C^0([0,T];L^2_w(\dom))$:
\begin{equation}\label{3.C0Bw}
\begin{aligned}
  C^0([0,T];B_w) &= \mbox{set of all weakly continuous functions }
	u:[0,T]\to L^2(\dom)\\[-1mm]
  &\phantom{xx}\mbox{ such that }
  \textstyle\sup_{t\in(0,T)}\|u(t)\|_{L^2(\dom)}\le r.
\end{aligned}
\end{equation}
This space is metrizable with the metric
$q^*(u,v)=\sup_{t\in(0,T)}q(u(t),v(t))$ \cite[Theorem 3.29]{Bre11}.
By the Banach-Alaoglu theorem, $B_w$ is compact \cite[Theorem 3.16]{Bre11},
so, $(C^0([0,T];B_w),q^*)$ is a complete metric space.

The following lemma ensures that any sequence in $C^0([0,T];B)$ which converges
in some space $C^0([0,T];U')$ with $U\subset H^1(\dom)$ is also convergent
in $C^0([0,T];B_w)$. We apply this lemma with $U=H^3(\dom)$.

\begin{lemma}[Lemma 2.1 in \cite{BrMo14}]\label{lem.aux}
Let $u_n:[0,T]\to L^2(\dom)$ ($n\in\N$) be functions satisfying
\begin{align*}
  & \sup_{n\in\N}\sup_{t\in(0,T)}\|u_n(t)\|_{L^2(\dom)}\le r, \\
  & u_n\to u \quad\mbox{in }C^0([0,T];U') \quad\mbox{as }n\to\infty,
\end{align*}
where $U\subset H^1(\dom)$ and $U'$ is the dual space of $U$.
Then $u_n$, $u\in C^0([0,T];B_w)$ and $u_n\to u$ in $C^0([0,T];B_w)$ as
$n\to\infty$.
\end{lemma}

We define the space
\begin{equation}\label{3.Z}
  Z_T:=C^0([0,T];H^3(\dom)')\cap L_w^2(0,T;H^1(\dom))\cap L^2(0,T;L^2(\dom))
	\cap C^0([0,T];L_w^2(\dom)),
\end{equation}
endowed with the topology $\mathcal{T}$ which is the maximum of the topologies
$\mathcal{T}_i$, $i=1,2,3,4$, of the corresponding spaces. On this space, we can
formulate a compactness criterion which is analogous to the result
due to Mikulevcius and Rozowskii \cite{MiRo05}. 

\begin{lemma}[Compactness criterion]\label{lem.comp}
Let $(Z_T,\mathcal{T})$ be as defined in \eqref{3.Z}. A set $K\subset Z_T$
is $\mathcal{T}$-relatively compact if the following three conditions hold:
\begin{enumerate}
\item $\sup_{u\in K}\sup_{t\in(0,T)}\|u(t)\|_{L^2(\dom)}<\infty$,
\item $K$ is bounded in $L^2(0,T;H^1(\dom))$, and
\item $\lim_{\delta\to 0}\sup_{u\in K}\sup_{s,t\in(0,T),\,|s-t|\le\delta}
\|u(s)-u(t)\|_{H^3(\dom)'}=0$.
\end{enumerate}
\end{lemma}

We refer to \cite[Lemma 2.3]{BrMo14} for a proof. The result follows since
the embeddings $H^1(\dom)\hookrightarrow L^2(\dom)\hookrightarrow H^3(\dom)'$
are continuous and the embedding $H^1(\dom)\hookrightarrow L^2(\dom)$ is compact,
such that we can apply Dubinskii's Theorem \cite{Dub65} (also see \cite{Sim87})
to a sequence $(u_n)_{n\in\N}\subset K$
to conclude that there exists a subsequence of $(u_n)_{n\in\N}$ that is convergent in
$C^0([0,T];H^3(\dom)')$. By Lemma \ref{lem.aux}, this subsequence is also convergent
in $C^0([0,T];B_w)$. 

The compactness criterion in Lemma \ref{lem.comp} allows for a proof of the
following tightness criterion taken from \cite[Corollary 2.6]{BrMo14}.

\begin{theorem}[Tightness criterion]\label{thm.tight}
Let $H$, $V$, and $U$ be separable Hilbert spaces such that the embeddings
$U\hookrightarrow V\hookrightarrow H$ are dense and continuous and the
embedding $V\hookrightarrow H$ is compact. Furthermore, let
$(X_n)_{n\in\N}$ be a sequence of continuous $\Fil$-adapted $U'$-valued
stochastic processes such that
\begin{enumerate}
\item there exists $C>0$ such that
$$
  \sup_{n\in\N}\E\bigg(\sup_{t\in(0,T)}\|X_n(t)\|_H^2\bigg) \le C,
$$
\item there exists $C>0$ such that
$$
  \sup_{n\in\N}\E\bigg(\int_0^T\|X_n(t)\|_V^2 dt\bigg) \le C,
$$
\item $(X_n)_{n\in\N}$ satisfies the Aldous condition in $U'$
(see Definition \ref{def.ald} in the appendix).
\end{enumerate}
Furthermore, let $\Prob_n$ be the law of $X_n$ on $Z_T$. Then $(\Prob_n)_{n\in\N}$ 
is tight on $Z_T$.
\end{theorem}

The main result of this subsection is the tightness of the laws 
$\Law(u^{(N)})$ of the solutions $u^{(N)}$
to \eqref{2.approx1}-\eqref{2.approx2}. 

\begin{lemma}\label{lem.ut}
The set of measures $\{\Law(u^{(N)}):N\in\N\}$ is tight on $(Z_T,\mathcal{T})$.
\end{lemma}

\begin{proof}
The idea of the proof is to apply Theorem \ref{thm.tight} with
$U=H^3(\dom)$, $V=H^1(\dom)$, and $H=L^2(\dom)$. 
In view of estimates \eqref{3.est1} and \eqref{3.est2}, conditions (1) and (2)
of Theorem \ref{thm.tight} are fulfilled. It remains to show that
$(u^{(N)})_{N\in\N}$ satisfies the Aldous condition in $H^3(\dom)'$. To this end,
let $(\tau_N)_{N\in\N}$ be a sequence of $\Fil$-stopping times such 
that $0\le\tau_N\le T$. Let $t\in[0,T]$ and $\phi\in H^3(\dom)$. 
Then \eqref{2.approx1} can be written as
\begin{align}
  \langle u_i^{(N)}(t),\phi\rangle
	&= \langle \Pi_N(u_i^0),\phi\rangle
	- \sum_{j=1}^n\int_0^t\big\langle A_{ij}(u^{(N)})\na u_j^{(N)},\na\Pi_N\phi
	\big\rangle ds \nonumber \\
	&\phantom{xx}{}
	+ \sum_{j=1}^n\bigg\langle\int_0^t\Pi_N\big(\sigma_{ij}(u^{(N)}(s))\big)dW_j(s),\phi
	\bigg\rangle \nonumber \\
  &=: J_1^{(N)} + J_2^{(N)}(t) + J_3^{(N)}(t), \label{3.J}
\end{align}
where $\langle\cdot,\cdot\rangle$ is the dual pairing between 
$H^3(\dom)'$ and $H^3(\dom)$.
We estimate each term on the right-hand side individually. 

First, consider the term involving the diffusion coefficients. Let $\theta>0$. Then,
using the (at most) quadratic dependence of $A_{ij}$ on $u_k$ and the
continuous embedding $H^3(\dom)\hookrightarrow W^{1,\infty}(\dom)$
(this is another instance where we use $d\le 3$), we find that
\begin{align*}
  \E\bigg|&\int_{\tau_N}^{\tau_N+\theta}
	\big\langle A_{ij}(u^{(N)})\na u_j^{(N)},\na\Pi_N\phi\big\rangle ds\bigg| \\
	&\le \E\int_{\tau_N}^{\tau_N+\theta}\|A_{ij}(u^{(N)})\|_{L^2(\dom)}
	\|\na u_j^{(N)}\|_{L^2(\dom)}\|\na\phi\|_{L^\infty(\dom)} ds \\
	&\le \E\bigg(\int_{\tau_N}^{\tau_N+\theta}\big(1+\|(u^{(N)})^2\|_{L^2(\dom)}\big)
	\|\na u^{(N)}\|_{L^2(\dom)}ds\bigg)\|\phi\|_{H^3(\dom)} \\
	&\le \E\bigg(\big(\theta^{1/2} + \theta^{1/6}\|(u^{(N)})^2\|_{L^3(0,T;L^2(\dom))}\big)
	\|\na u^{(N)}\|_{L^2(0,T;L^2(\dom))}\bigg)\|\phi\|_{H^3(\dom)} \\
	&\le \bigg\{\theta^{1/2}
	+ \theta^{1/6}\bigg(\E\bigg(\int_0^T\|(u^{(N)})^2\|_{L^2(\dom)}^3 dt\bigg)^{2/3}
	\bigg)^{1/2} \\
	&\phantom{xx}{}\times\bigg(\E\int_0^T\|\na u^{(N)}\|_{L^2(\dom)}^2 dt\bigg)^{1/2}\bigg\}
	\|\phi\|_{H^3(\dom)},
\end{align*}
where in the last two inequalities we applied the H\"older inequality with respect to
time and then with respect to the random variable. The vector 
$(u^{(N)}))^2$ consists of elements $(u_i^{(N)}))^2$ for $i=1,\ldots,N$. 
Taking into account the estimates
from Lemmas \ref{lem.est1} and \ref{lem.3}, we deduce that
\begin{equation}\label{3.theta1}
  \E\bigg|\int_{\tau_N}^{\tau_N+\theta}
	\big\langle A_{ij}(u^{(N)})\na u_j^{(N)},\na\Pi_N\phi\big\rangle ds\bigg| 
	\le C\theta^{1/6}\|\phi\|_{H^3(\dom)}. 
\end{equation}
For the stochastic term, we use assumption \eqref{1.sigma} on $\sigma$,
the It\^{o} isometry (see Proposition \ref{prop.iso}), 
and the H\"older inequality to obtain
\begin{align}
  \E\bigg|\bigg\langle &\int_{\tau_N}^{\tau_N+\theta}
	\Pi_N(\sigma_{ij}(u^{(N)}(s)))dW_j(s),\phi\bigg\rangle\bigg|^2 \nonumber \\
	&\le \E\bigg(\int_{\tau_N}^{\tau_N+\theta}\|\sigma(u^{(N)})\|_{\T_2(Y;L^2(\dom))}^2
	dt\bigg)\|\phi\|_{L^2(\dom)}^2 \nonumber \\
	&\le C_\sigma\E\bigg(\int_{\tau_N}^{\tau_N+\theta}\big(1+\|u^{(N)}\|_{L^2(\dom)}^2
	\big)dt\bigg)\|\phi\|_{L^2(\dom)}^2 \nonumber \\
	&\le C\bigg(\theta + \theta^{1/3}\bigg(\E\int_0^T\|u^{(N)}\|_{L^2(\dom)}^3 dt
	\bigg)^{2/3}\bigg)\|\phi\|_{L^2(\dom)}^2 
  \le C\theta^{1/3}\|\phi\|_{L^2(\dom)}^2. \label{3.theta2}
\end{align}

Next, let $\kappa>0$ and $\eps>0$. By the definition of the $H^3(\dom)'$ norm, the 
Chebyshev inequality, and estimate \eqref{3.theta1}, we have
\begin{align*}
  \Prob\Big\{ & \|J_2^{(N)}(\tau_N+\theta)-J_2^{(N)}(\tau_N)\|_{H^3(\dom)'}\ge\kappa\Big\}
  \le \frac{1}{\kappa}\E\|J_2^{(N)}(\tau_N+\theta)-J_2^{(N)}(\tau_N)\|_{H^3(\dom)'} \\
	&= \frac{1}{\kappa}\sup_{\|\phi\|_{H^3(\dom)}=1}\E
	\big|\big\langle J_2^{(N)}(\tau_N+\theta)-J_2^{(N)}(\tau_N),\phi\big\rangle\big|
	\le \frac{C\theta^{1/6}}{\kappa}.
\end{align*}
Thus, choosing $\delta_1=(\kappa\eps/C)^6$, we infer that
$$
  \sup_{N\in\N}\sup_{0<\theta<\delta_1}
	\Prob\Big\{\|J_2^{(N)}(\tau_N+\theta)-J_2^{(N)}(\tau_N)
	\|_{H^3(\dom)'}\ge\kappa\Big\} \le \eps.
$$
In a similar way, it follows that
\begin{align*}
  \Prob\Big\{\|J_3^{(N)}(\tau_N+\theta)-J_3^{(N)}(\tau_N)\|_{H^3(\dom)'}\ge\kappa\Big\}
  &\le \frac{1}{\kappa^2}\E\|J_3^{(N)}(\tau_N+\theta)
	-J_3^{(N)}(\tau_N)\|_{H^3(\dom)'}^2 \\
	&\le \frac{C_2\theta^{1/3}}{\kappa^2},
\end{align*}
and choosing $\delta_2=(\kappa^2\eps/C)^3$ gives
$$
  \sup_{N\in\N}\sup_{0<\theta<\delta_1}
	\Prob\Big\{\|J_3^{(N)}(\tau_N+\theta_2)-J_3^{(N)}(\tau_N)
	\|_{H^3(\dom)'}\ge\kappa\Big\} \le \eps.
$$
This shows that the Aldous condition holds for all three terms $J_i^{(N)}$,
$i=1,2,3$. Consequently, in view of \eqref{3.J},
it also holds for $(u^{(N)})_{N\in\N}$. We conclude the proof
by invoking Theorem \ref{thm.tight}.
\end{proof}


\subsection{Convergence of the approximate solutions}\label{sec.conv}

First, we show that the space $Z_T$, defined in \eqref{3.Z}, verifies the
assumption of the Skorokhod-Jakubowski theorem (see Theorem \ref{thm.skoro} in the
appendix). More precisely, we prove that on each space in definition \eqref{3.Z}
of $Z_T$, there exists a countable set of continuous real-valued functions
separating points. 

\begin{lemma}\label{lem.Z}
The topological space $Z_T$, defined in \eqref{3.Z}, satisfies the assumption of
Theorem \ref{thm.skoro}.
\end{lemma}

\begin{proof}
Since the spaces $C^0([0,T];H^3(\dom)')$ and $L^2(0,T;L^2(\dom))$
are separable, metrizable, and complete, the assumption of Theorem \ref{thm.skoro} 
is satisfied; see \cite[Expos\'e 8]{Bad70}.
For the space $L_w^2(0,T;H^1(\dom))$, it is sufficient to define
$$
  f_m(u) = \int_0^T (u(t),v_m(t))_{H^1(\dom)}dt\in\R, \quad 
	\mbox{where }u\in L_w^2(0,T;H^1(\dom)),\ m\in\N,
$$
and $(v_m)_{m\in\N}$ is a dense subset of $L^2(0,T;H^1(\dom))$. 

It remains to consider the space $C^0([0,T];L_w^2(\dom))$. Let $(w_m)_{m\in\N}$
be a dense subset of $L^2(0,T;L^2(\dom))$ and let $\mathbb{Q}_T$ be the set of
rational numbers from the interval $[0,T]$. Then the family
$\{f_{m,t}:m\in\N,$ $t\in\mathbb{Q}_T\}$, defined by
$$
  f_{m,t}(u) = (u(t),w_m)_{L^2(\dom)}\in\R, \quad\mbox{where }
	u\in C^0([0,T];L_w^2(\dom)),\ m\in\N,\ t\in\mathbb{Q}_T,
$$
consists of continuous functions separating points in $C^0([0,T];L_w^2(\dom))$.
\end{proof}

In view of Lemma \ref{lem.Z} and Theorem \ref{thm.skoro}, we infer the following result.

\begin{corollary}\label{coro.Z}
Let $(\eta_n)_{n\in\N}$ be a sequence of $Z_T$-valued random variables such that
their laws $\Law(\eta_n)$ on $(Z_T,\mathcal{T})$ form a tight sequence of probability
measures. Then there exists a subsequence $(\eta_k)_{k\in\N}$, which is not
relabeled, a probability space $(\widetilde\Omega,\widetilde\F,\widetilde\Prob)$,
and $Z_T$-valued random variables $\widetilde\eta$, $\widetilde\eta_k$ with $k\in\N$
such that the variables $\eta_k$ and $\widetilde\eta_k$ have the same laws on $Z_T$
and $(\widetilde\eta_k)_{k\in\N}$ converges to $\widetilde\eta$ a.s.\ on 
$\widetilde\Omega$.
\end{corollary}

By Lemma \ref{lem.ut}, the set of measures $\{\Law(u^{(N)}):N\in\N\}$ is tight
on $(Z_T,\mathcal{T})$ and by Lemma \ref{lem.Z}, the space $Z_T\times C^0([0,T];Y_0)$
satisfies the assumption of Theorem \ref{thm.skoro}. 
Therefore, we can apply
Corollary \ref{coro.Z} to deduce the existence of a subsequence of $(u^{(N)})_{N\in\N}$,
which is not relabeled, a probability space 
$(\widetilde\Omega,\widetilde\F,\widetilde\Prob)$, and, on this space,
$Z_T\times C^0([0,T];Y_0)$-valued random variables $(\widetilde u,\widetilde W)$,
$(\widetilde u^{(N)},\widetilde W^{(N)})$ with $N\in\N$ such that
$(\widetilde u^{(N)},\widetilde W^{(N)})$ has the same law as $(u^{(N)},W)$
on $\B(Z_T\times C^0([0,T];Y_0))$ and
$$
  (\widetilde u^{(N)},\widetilde W^{(N)}) \to (\widetilde u,\widetilde W)
	\quad\mbox{in }Z_T\times C^0([0,T];Y_0),\ \widetilde\Prob\mbox{-a.s., as }N\to\infty.
$$
Because of the definition of the space $Z_T$, this convergence means that
$\widetilde\Prob$-a.s.,
\begin{align}
  \widetilde u^{(N)}\to \widetilde u &\quad\mbox{in }C^0([0,T];H^3(\dom)'), \nonumber \\
	\widetilde u^{(N)}\rightharpoonup \widetilde u &\quad\mbox{weakly in }
	L^2(0,T;H^1(\dom)), \nonumber \\
	\widetilde u^{(N)}\to \widetilde u &\quad\mbox{in }L^2(0,T;L^2(\dom)), \label{3.conv} \\
	\widetilde u^{(N)}\to \widetilde u &\quad\mbox{in }C^0([0,T];L_w^2(\dom)), \nonumber \\
  \widetilde W^{(N)}\to \widetilde W &\quad\mbox{in }C^0([0,T];Y_0). \nonumber
\end{align}

Since $u^{(N)}$ is an element of $C^0([0,T];H_N)$ $\Prob$-a.s., $C^0([0,T];H_N)$
is a Borel set of $C^0([0,T];$ $H^3(\dom)')\cap L^2(0,T;L^2(\dom))$, and since
$u^{(N)}$ and $\widetilde u^{(N)}$ have the same laws, we infer that
$$
  \Law(\widetilde u^{(N)})\big(C^0([0,T];H_N)\big) = 1 \quad\mbox{for all }N\ge 1.
$$
Note that, as $\B(Z_T\times C^0([0,T];Y_0))$ is a subset of
$\B(Z_T)\times\B(C^0([0,T];Y_0))$, the function $\widetilde u$ is a $Z_T$-Borel
random variable. Furthermore, in view of estimates \eqref{3.est1}-\eqref{3.est3}
and \eqref{3.estp} and the equivalence of the laws of $\widetilde u^{(N)}$
and $\widetilde u$ on $\B(Z_T)$, we have the uniform bounds
\begin{align}
  \sup_{N\in\N}\widetilde\E\Big(\sup_{t\in(0,T)}\|\widetilde u^{(N)}\|_{L^2(\dom)}^2\Big)
	&\le C_1, \label{3.tilde1} \\
  \sup_{N\in\N}\widetilde\E\bigg(\int_0^T\|\widetilde u^{(N)}\|_{H^1(\dom)}^2 dt\bigg)
	+ \alpha\sup_{N\in\N}\widetilde\E\bigg(\int_0^T\big\|(\widetilde u^{(N)})^2
	\big\|_{H^1(\dom)}^2 dt\bigg)	&\le C_1, \label{3.tilde2} \\
  \sup_{N\in\N}\widetilde\E\bigg(\sup_{t\in(0,T)}
	\|\widetilde u^{(N)}\|_{L^2(\dom)}^p\bigg) 
	&\le C_2, \label{3.tilde3}
\end{align}
where $p\ge 2$ is any number.

We deduce from \eqref{3.tilde2} that there exists a subsequence of 
$(\widetilde u^{(N)})$ (not relabeled) which is weakly converging in
$L^2((0,T)\times\widetilde\Omega;H^1(\dom))$ as $N\to\infty$. Since
$\widetilde u^{(N)}\to\widetilde u$ $\widetilde\Prob$-a.s.\ in $Z_T$,
we conclude that $\widetilde u\in L^2((0,T)\times\widetilde\Omega;H^1(\dom))$, i.e.
\begin{equation}\label{3.uH1}
  \widetilde\E\int_0^T\|\widetilde u(t)\|_{H^1(\dom)}^2 dt < \infty.
\end{equation}
Similarly, the bound \eqref{3.tilde1} allows us to extract a subsequence
which is weakly* convergent in $L^2(\widetilde\Omega;L^\infty(0,T;L^2(\dom)))$ and
\begin{equation}\label{3.uL2}
  \widetilde\E\bigg(\sup_{t\in(0,T)}\|\widetilde u(t)\|_{L^2(\dom)}^2\bigg) < \infty.
\end{equation}

The convergence $\widetilde u^{(N)}\to \widetilde u$ in $L^2(0,T;L^2(\dom))$ 
$\widetilde\Prob$-a.s.\ implies, up to a subsequence, that
$$
  \widetilde u^{(N)}\to\widetilde u \quad\mbox{a.e. in }\dom,
	\widetilde\Prob\mbox{-a.s.}
$$
In particular, we have (componentwise) $(\widetilde u^{(N)})^2\to(\widetilde u)^2$
a.e.\ in $\dom$, $\widetilde\Prob$-a.s. 
On the other hand, by estimate \eqref{3.tilde2}, there exists a subsequence
of $((\widetilde u^{(N)})^2)_{N\in\N}$ weakly converging to some function
$v$ in $L^2(\widetilde\Omega;
L^2(0,T;H^1(\dom)))$. The uniqueness of the limit function then implies
that $v=\widetilde u^2$ and consequently,
$$
  (\widetilde u^{(N)})^2\rightharpoonup (\widetilde u)^2
	\quad\mbox{weakly in }L^2(\widetilde\Omega;L^2(0,T;H^1(\dom))).
$$
It remains to show that the stochastic process $\widetilde u$ is a martingale
solution to \eqref{1.eq}. 
The following lemmas are taken from \cite[Lemma 5.2 and proof]{BGJ13}.

\begin{lemma}\label{lem.W1}
Suppose that the process $(\widetilde W^{(N)}(t))_{t\in[0,T]}$, defined on
$(\widetilde\Omega,\widetilde\F,\widetilde\Prob)$, has the same law as the
$Y$-valued cylindrical Wiener process $W$, defined on $(\Omega,\F,\Prob)$.
Then $\widetilde W^{(N)}$ is also a $Y$-valued cylindrical Wiener process on
$(\widetilde\Omega,\widetilde\F,\widetilde\Prob)$.
\end{lemma}

\begin{lemma}\label{lem.W2}
The process $(\widetilde W(t))_{t\in[0,T]}$ is a $Y$-valued cylindrical Wiener
process on $(\widetilde\Omega,\widetilde\F,\widetilde\Prob)$. If $0\le s<t\le T$,
the increments $\widetilde W(t)-\widetilde W(s)$ are independent of the
$\sigma$-algebra generated by $\widetilde u(r)$ and $\widetilde W(r)$ for
$r\in[0,s]$.
\end{lemma}

We denote by $\widetilde\Fil$ the filtration generated by $(\widetilde u,\widetilde W)$
and by $\widetilde\Fil^{(N)}$ the filtration generated by 
$(\widetilde u^{(N)},\widetilde W^{(N)})$. Lemma \ref{lem.W1} implies that
$\widetilde u$ is progressively measurable with respect to $\widetilde\Fil$, and
Lemma \ref{lem.W2} shows that $\widetilde u^{(N)}$ is progressively measurable
with respect to $\widetilde\Fil^{(N)}$. 

The following lemma plays a significant role in establishing the existence of a
martingale solution to \eqref{1.eq}. 

\begin{lemma}\label{lem.E}
It holds for all $s$, $t\in[0,T]$ with $s\le t$ and all $\phi_1\in L^2(\dom)$
and $\phi_2\in H^3(\dom)$ satisfying $\na\phi_2\cdot\nu=0$ on $\pa\dom$ that
\begin{align}
  \lim_{N\to\infty}\widetilde\E\int_0^T\big(\widetilde u^{(N)}(t)-\widetilde u(t),
	\phi_1\big)_{L^2(\dom)}^2dt &= 0, \label{3.E1} \\
	\lim_{N\to\infty}\widetilde\E\big(\widetilde u^{(N)}(0)-\widetilde u(0),
	\phi_1\big)_{L^2(\dom)}^2 &= 0, \label{3.E2} \\
	\lim_{N\to\infty}\widetilde\E\int_0^T\bigg|\sum_{j=1}^n\int_0^t\Big\langle
	A_{ij}(\widetilde u^{(N)}(s))\na\widetilde u_j^{(N)}(s)
	- A_{ij}(\widetilde u(s))\na\widetilde u_j(s),\na\phi_2\Big\rangle
	ds\bigg|dt &= 0, \label{3.E3} \\
	\lim_{N\to\infty}\widetilde\E\int_0^T\bigg|\sum_{j=1}^n\int_0^t\Big(
	\sigma_{ij}(\widetilde u^{(N)}(s))d\widetilde W_j^{(N)}(s)-\sigma_{ij}(\widetilde u(s))
	d\widetilde W_j(s),\phi_1\Big)_{L^2(\dom)}\bigg|^2 dt &= 0. \label{3.E4}
\end{align}
\end{lemma}

\begin{proof}
Let $\phi_1\in L^2(\dom)$. We know that $\widetilde u^{(N)}\to \widetilde u$ in
$Z_T$ $\widetilde\Prob$-a.s. In particular, $\widetilde u^{(N)}\to \widetilde u$
in $C^0([0,T];L_w^2(\dom))$ $\widetilde\Prob$-a.s., which means that
for any $t\in[0,T]$,
$$
  \lim_{N\to\infty}(\widetilde u^{(N)}(t),\phi_1)_{L^2(\dom)}
	= (\widetilde u(t),\phi_1)_{L^2(\dom)} \quad\widetilde\Prob\mbox{-a.s.}
$$
Estimate \eqref{3.tilde1} provides a uniform bound for 
$(\widetilde u^{(N)}(t),\phi_1)_{L^2(\dom)}^2$ such that we can apply the dominated
convergence theorem to conclude that
\begin{equation}\label{3.aux}
  \lim_{N\to\infty}\int_0^T\big(\widetilde u^{(N)}(t)-\widetilde u(t),
	\phi_1\big)_{L^2(\dom)}^2 dt = 0 \quad\widetilde\Prob\mbox{-a.s.}
\end{equation}
We have for any $r>1$, by \eqref{3.tilde3},
$$
  \widetilde\E\bigg(\bigg|\int_0^T\|\widetilde u^{(N)}(t)-\widetilde u(t)\|_{L^2(\dom)}^2
	dt\bigg|^r\bigg)
	\le C\widetilde\E\int_0^T\big(\|\widetilde u^{(N)}(t)\|_{L^2(\dom)}^{2r}
	+ \|\widetilde u(t)\|_{L^2(\dom)}^{2r}\big)dt	\le C.
$$
This bound provides the equi-integrability of 
$\int_0^T\big(\widetilde u^{(N)}(t)-\widetilde u(t),\phi_1\big)_{L^2(\dom)}^2 dt$.
Taking into account the convergence \eqref{3.aux}, 
Vitali's convergence theorem (see the appendix) then shows that \eqref{3.E1} holds. 

Convergence \eqref{3.E2} follows in a similar way. Indeed, since
$\widetilde u^{(N)}\to\widetilde u$ in $C^0([0,T];L_w^2(\dom))$ 
$\widetilde\Prob$-a.s.\ and $\widetilde u$ is continuous at $t=0$, we infer that
for any $\phi_1\in L^2(\dom)$,
$$
  \lim_{N\to\infty}(\widetilde u^{(N)}(0),\phi_1)_{L^2(\dom)}
	= (\widetilde u(0),\phi_1)_{L^2(\dom)} \quad\widetilde\Prob\mbox{-a.s.}
$$
Then convergence \eqref{3.E2} follows from \eqref{3.tilde1} and Vitali's
convergence theorem. 

Next, we establish convergence \eqref{3.E3} through several steps.
Due to the structure of $A_{ij}(\widetilde u^{(N)})$, we need to show the
following three convergences:
\begin{align}
	\lim_{N\to\infty}\int_0^t\big\langle\na\widetilde u_j^{(N)}(s)-\na\widetilde u_j(s),
	\na\phi\big\rangle ds &= 0, \label{3.N1} \\
		\lim_{N\to\infty}\int_0^t\Big\langle \widetilde u_j^{(N)}(s)\widetilde u_k^{(N)}(s)
	\na\widetilde u_k^{(N)}(s) - \widetilde u_j(s)\widetilde u_k(s)\na\widetilde u_k(s),
	\na\phi\Big\rangle ds &= 0, \label{3.N2} \\
	\lim_{N\to\infty}\int_0^t\Big\langle(\widetilde u_k^{(N)}(s))^2
	\na\widetilde u_j^{(N)}(s) - (\widetilde u_k(s))^2\na\widetilde u(s),\na\phi
	\Big\rangle ds &= 0, \label{3.N3}
\end{align}
for $j\neq k$ and suitable test functions $\phi$. 
We deduce from convergence \eqref{3.conv} that \eqref{3.N1}
follows for all $\phi\in H^1(\dom)$. The second convergence \eqref{3.N2} is proved
as follows:
\begin{align*}
  \bigg|&\int_0^t\Big\langle\widetilde u_j^{(N)}(s)\widetilde u_k^{(N)}(s)
	\na\widetilde u_k^{(N)}(s) - \widetilde u_j(s)\widetilde u_k(s)\na\widetilde u_k(s),
	\na\phi\Big\rangle ds\bigg| \\
	&= \frac12\bigg|\int_0^t\Big\langle \widetilde u_j^{(N)}(s)
	\na\big(\widetilde u_k^{(N)}(s)\big)^2
	- \widetilde u_j(s)\na\big(\widetilde u_k(s)\big)^2,\na\phi\Big\rangle ds\bigg| \\
	&= \frac12\bigg|\int_0^t\Big\langle\big(\widetilde u_j^{(N)}(s)-\widetilde u_j(s)\big)
	\na\big(\widetilde u_k^{(N)}(s)\big)^2 
	+ \widetilde u_j(s)\na\big\{\big(\widetilde u_k^{(N)}(s)\big)^2
	- \big(\widetilde u_k(s)\big)^2\big\},\na\phi\Big\rangle ds\bigg| \\
	&\le \frac12\int_0^t\|\widetilde u_j^{(N)}(s)-\widetilde u_j(s)\|_{L^2(\dom)}
	\|\na(\widetilde u_k^{(N)})^2\|_{L^2(\dom)}\|\na\phi\|_{L^\infty(\dom)} ds \\
	&\phantom{xx}{}+ \frac12\bigg|\int_0^t\Big(\widetilde u_j(s)
	\na\big\{\big(\widetilde u_k^{(N)}(s)\big)^2
	- \big(\widetilde u_k(s)\big)^2\big\},\na\phi\Big)_{L^2(\dom)} 
	ds\bigg| \\
	&=: I^{(N)}_1 + I^{(N)}_2.
\end{align*}
Let $\phi\in H^3(\dom)$. Then the embedding $H^3(\dom)\hookrightarrow W^{1,\infty}(\dom)$
is continuous for $d\le 3$ and, using the Cauchy-Schwarz inequality,
$$
  I^{(N)}_1 \le \frac12\|\phi\|_{H^3(\dom)}\|\widetilde u_j^{(N)}
	-\widetilde u_j\|_{L^2(0,T;L^2(\dom))}
	\|\na(\widetilde u_k^{(N)})^2\|_{L^2(0,T;L^2(\dom))}.
$$
Since $\widetilde u^{(N)}\to \widetilde u$ in $L^2(0,T;L^2(\dom))$ $\widetilde\Prob$-a.s.
and $\na(\widetilde u^{(N)})^2$ is uniformly bounded in $L^2(0,T;$ $L^2(\dom))$,
it follows that $I^{(N)}_1\to 0$ as $N\to\infty$. For the second integral, we
observe that $\widetilde u_j\na\phi\in L^2(0,T;L^2(\dom))$ (using \eqref{3.tilde2})
and $(\widetilde u^{(N)})^2\rightharpoonup(\widetilde u)^2$ weakly in
$L^2(0,T;H^1(\dom))$ (by \eqref{3.conv}). This implies that $I^{(N)}_2\to 0$
as $N\to\infty$, and we have proved \eqref{3.N2}.

We turn to the proof of \eqref{3.N3}. Let $\phi\in H^3(\dom)$ be such that
$\na\phi\cdot\nu=0$ on $\pa\dom$. An integration by parts leads to
\begin{align*}
  \int_0^t&\Big\langle(\widetilde u_k^{(N)}(s))^2
	\na\widetilde u_j^{(N)}(s) - (\widetilde u_k(s))^2\na\widetilde u_j(s),\na\phi
	\Big\rangle ds \\
	&= \int_0^t\int_\dom\Big((\widetilde u_k^{(N)}(s))^2\na\widetilde u_j^{(N)}(s)
	- (\widetilde u_k(s))^2\na\widetilde u_j(s)\Big)\cdot\na\phi dxds \\
	&= -\int_0^t\int_\dom\Big((\widetilde u_k^{(N)}(s))^2\widetilde u_j^{(N)}(s)
	- (\widetilde u_k(s))^2\widetilde u_j(s)\Big)\Delta\phi dxds \\
	&\phantom{xx}{}- \int_0^t\int_\dom\Big(\widetilde u_j^{(N)}(s)\na(u_k^{(N)}(s))^2
	- \widetilde u_j(s)\na(\widetilde u_k(s))^2\Big)\cdot\na\phi dxds \\
	&=: I^{(N)}_3 + I_4^{(N)}.
\end{align*}
The estimates for $I^{(N)}_1+I_2^{(N)}$ show that $I^{(N)}_4\to 0$
as $N\to\infty$. We estimate $I_3^{(N)}$ as follows, using the continuous
embeddings $H^3(\dom)\hookrightarrow W^{2,4}(\dom)$ and
$H^1(\dom)\hookrightarrow L^4(\dom)$ (for $d\le 3$):
\begin{align*}
  I_3^{(N)} &= -\int_0^t\int_\dom\big(\widetilde u_j^{(N)}(s)-\widetilde u_j(s)\big)
	\big(\widetilde u_k^{(N)}(s)\big)^2\Delta\phi dxds \\
	&\phantom{xx}{}+ \int_0^t\int_\dom
	\big((\widetilde u_k^{(N)}(s))^2-(\widetilde u_k(s))^2\big)\widetilde u_j(s)
	\Delta\phi dxds \\
	&\le \int_0^t\big\|\widetilde u_j^{(N)}(s)-\widetilde u_j(s)\big\|_{L^2(\dom)}
  \big\|(\widetilde u_k^{(N)}(s))^2\big\|_{L^4(\dom)}\|\Delta\phi\|_{L^4(\dom)}ds  \\
	&\phantom{xx}+ \int_0^t\int_\dom
	\big((\widetilde u_k^{(N)}(s))^2-(\widetilde u_k(s))^2\big)\widetilde u_j(s)
	\Delta\phi dxds \\
  &\le \big\|\widetilde u_j^{(N)}-\widetilde u_j\big\|_{L^2(0,T;L^2(\dom))}
	\big\|(\widetilde u_k^{(N)})^2\big\|_{L^2(0,T;H^1(\dom))}\|\phi\|_{H^3(\dom)} \\
	&\phantom{xx}+ \int_0^t\int_\dom
	\big((\widetilde u_k^{(N)}(s))^2-(\widetilde u_k(s))^2\big)\widetilde u_j(s)
	\Delta\phi dxds.
\end{align*}
The convergences \eqref{3.conv} and $\widetilde u_j\Delta\phi\in L^2(0,T;L^2(\dom))$
$\widetilde\Prob$-a.s.\ imply that $I_3^{(N)}\to 0$ as $N\to\infty$. 

Convergences \eqref{3.N1}-\eqref{3.N3} imply that $\widetilde\Prob$-a.s.
\begin{equation}\label{3.A1}
  \lim_{N\to\infty}\int_0^t\big(A_{ij}(\widetilde u^{(N)}(s))\na\widetilde u_j^{(N)}(s),
	\na \phi_2\big)_{L^2(\dom)} ds
	= \int_0^t\big(A_{ij}(\widetilde u(s))\na \widetilde u_j(s),
	\na\phi_2\big)_{L^2(\dom)} ds
\end{equation}
for all $\phi_2\in H^3(\dom)$ with $\na\phi_2\cdot\nu=0$ on $\pa\dom$. Furthermore,
employing the structure of $A_{ij}(u^{(N)})$, the continuous embedding
$H^3(\dom)\hookrightarrow W^{1,\infty}(\dom)$ (again for $d\le 3$ only), and 
estimates \eqref{3.tilde2}-\eqref{3.tilde3}, we find that
\begin{align*}
  \widetilde\E&\bigg(\bigg|\int_0^t\big(A_{ij}(\widetilde u^{(N)}(s))
	\na\widetilde u_j^{(N)}(s),\na\phi_2\big)_{L^2(\dom)}ds\bigg|^2\bigg) \\
  &\le \|\na\phi_2\|^2_{L^\infty(\dom)}\widetilde\E\bigg(\bigg|\int_0^t
	\big\|A_{ij}(u^{(N)}(s))\na\widetilde u_j^{(N)}(s)\big\|_{L^1(\dom)} 
	ds\bigg|^2\bigg) \\
	&\le C\|\phi_2\|^2_{H^3(\dom)}\widetilde\E\bigg(\bigg|\int_0^t
	\big(1+\|\widetilde u^{(N)}(s)^2\|_{L^2(\dom)}\big)
	\|\na\widetilde u^{(N)}(s)\|_{L^2(\dom)}ds\bigg|^2\bigg) \\
	&\le C\|\phi_2\|^2_{H^3(\dom)}\Big(T^{1/2}\big(
	\widetilde\E\|\widetilde u^{(N)}\|_{L^2(0,T;H^1(\dom))}^2\big)^{1/2} \\
	&\phantom{xx}{}
	+ T^{1/6}\big(\widetilde\E\|(\widetilde u^{(N)})^2\|_{L^3(0,T;L^2(\dom))}^3\big)^{1/3}
	\big(\widetilde\E\|\widetilde u^{(N)}\|_{L^2(0,T;H^1(\dom))}^2\big)^{1/2}\Big)
	\le C.
\end{align*}
This bound and the $\widetilde\Prob$-a.s.\ convergence \eqref{3.A1} allow us to
apply the Vitali convergence theorem to infer that
\eqref{3.E3} holds.

It remains to prove convergence \eqref{3.E4}. Since $\widetilde{W}^{(N)} 
\to \widetilde{W}$ in $C^0([0,T]; Y_0)$, it is sufficient to show that 
$\sigma_{ij}(\widetilde{u}^{(N)}) \to \sigma_{ij}(\widetilde{u})$ in 
$L^2(0,T; \T_2(Y; L^2(\dom)))$ $\Prob$-a.s. We estimate for $\phi\in L^2(\dom)$:
\begin{align*}
  \int_0^t&\big\|\big(\sigma_{ij}(\widetilde u^{(N)}(s))-\sigma_{ij}(\widetilde u(s)),
	\phi\big)_{L^2(\dom)}\big\|_{\T_2(Y;\R)}^2 ds \\
	&\le \int_0^t\big\|\sigma_{ij}(\widetilde u^{(N)}(s))-\sigma_{ij}(\widetilde u(s))
	\big\|_{\T_2(Y;L^2(\dom))}^2\|\phi\|_{L^2(\dom)}^2 ds \\
	&\le C_\sigma\|\widetilde u^{(N)}-\widetilde u\|_{L^2(0,T;L^2(\dom))}^2
	\|\phi\|_{L^2(\dom)}^2.
\end{align*}
Since $\widetilde u^{(N)}\to\widetilde u$ in $L^2(0,T;L^2(\dom))$ 
$\widetilde\Prob$-a.s., by \eqref{3.conv}, we infer that for $t\in[0,T]$,
$\omega\in\widetilde\Omega$, and $\phi\in L^2(\dom)$,
\begin{equation}\label{3.aux5}
  \lim_{N\to\infty} 
	\int_0^t\big\|\big(\sigma_{ij}(\widetilde u^{(N)}(s))-\sigma_{ij}(\widetilde u(s)),
	\phi\big)_{L^2(\dom)}\big\|_{\T_2(Y;\R)}^2 ds = 0.
\end{equation}
We conclude from \eqref{3.tilde3} and \eqref{3.uH1} that
\begin{align*}
  \widetilde\E&\bigg|\int_0^t\big\|(\sigma_{ij}(\widetilde u^{(N)}(s))
	-\sigma_{ij}(\widetilde u(s)),\phi\big)_{L^2(\dom)}\big\|_{\T_2(Y;\R)}^2 ds
	\bigg|^2 \\
	&\le C\widetilde\E\bigg(\|\phi\|_{L^2(\dom)}^4\int_0^t\big(
	\|\sigma_{ij}(\widetilde u^{(N)}(s))\|_{\T_2(Y;L^2(\dom))}^4 
	+ \|\sigma_{ij}(\widetilde u(s))\|_{\T_2(Y;L^2(\dom))}^4\big)ds\bigg) \\
	&\le C\bigg(1+\widetilde\E\Big(\sup_{t\in(0,T)}\|\widetilde u^{(N)}(t)\|_{L^2(\dom)}^4
	+ \sup_{t\in(0,T)}\|\widetilde u(t)\|_{L^2(\dom)}^4\Big)\bigg) \le C.
\end{align*}
With this bound, convergence \eqref{3.aux5}, and the Vitali convergence theorem
we obtain for all $\phi\in L^2(\dom)$,
$$
  \lim_{N\to\infty}\widetilde\E\int_0^t\big\|\big(\sigma_{ij}(\widetilde u^{(N)}(s))
	- \sigma_{ij}(\widetilde u(s)),\phi\big)_{L^2(\dom)}\big\|_{\T_2(Y;\R)}^2 ds = 0.
$$
Hence, by the It\^{o} isometry (Proposition \ref{prop.iso})
for $t\in[0,T]$ and $\phi\in L^2(\dom)$, 
\begin{equation}\label{3.aux6}
  \lim_{N\to\infty}\widetilde\E\bigg|\bigg(\int_0^t
	\big(\sigma_{ij}(\widetilde u^{(N)}(s))-\sigma_{ij}(\widetilde u(s))\big)
	d\widetilde W_j(s),\phi\bigg)_{L^2(\dom)}\bigg|^2 = 0.
\end{equation}
We use the It\^{o} isometry again and estimates \eqref{3.tilde1} and \eqref{3.uL2}
for $N\in\N$, $t\in[0,T]$, and $\phi\in L^2(\dom)$ to infer that
\begin{align*}
  \widetilde\E&\bigg|\bigg(\int_0^t\big(\sigma_{ij}(\widetilde u^{(N)}(s))
	- \sigma_{ij}(\widetilde u(s))\big)d\widetilde W_j(s),\phi\bigg)_{L^2(\dom)}\bigg|^2 \\
	&= \widetilde\E\bigg(\int_0^t\big\|\big(\sigma_{ij}(\widetilde u^{(N)}(s))
	- \sigma_{ij}(\widetilde u(s)),\phi\big)_{L^2(\dom)}\big\|_{\T_2(Y;\R)}^2 ds
	\bigg) \\
	&\le \widetilde\E\bigg(\|\phi\|_{L^2(\dom)}^2\int_0^t\big\|
	\sigma_{ij}(\widetilde u^{(N)}(s))-\sigma_{ij}(\widetilde u(s))
	\big\|_{\T_2(Y;L^2(\dom))}^2 ds\bigg) \\
	&\le C\widetilde\E\bigg(t\sup_{t\in(0,T)}\|\widetilde u^{(N)}(t)\|_{L^2(\dom)}^2
	+ t\sup_{t\in(0,T)}\|\widetilde u(s)\|_{L^2(\dom)}^2\bigg)
	\le C.
\end{align*}
This bound and convergence \eqref{3.aux6} allow us to apply the dominated convergence
theorem to conclude that for all $\phi\in L^2(\dom)$,
$$
  \lim_{N\to\infty}\widetilde\E\int_0^T\bigg|\bigg(\int_0^t
	\big(\sigma_{ij}(\widetilde u^{(N)}(s))-\sigma_{ij}(\widetilde u(s))\big)
	d\widetilde W_j(s),\phi\bigg)_{L^2(\dom)}\bigg|^2 dt = 0.
$$
This shows \eqref{3.E4} and finishes the proof.
\end{proof}

Let us define
\begin{align*}
  \Lambda^{(N)}_i(\widetilde u^{(N)},\widetilde W^{(N)},\phi)(t)
	&:= (\Pi_N(\widetilde u_i(0)),\phi)_{L^2(\dom)} \\
	&\phantom{xx}{}
	+ \sum_{j=1}^n\int_0^t\Big\langle\Pi_N\diver\big(A_{ij}(\widetilde u^{(N)}(s))
	\na\widetilde u^{(N)}_j(s)\big),\phi\Big\rangle ds \\
	&\phantom{xx}{}+ \bigg(\sum_{j=1}^n\int_0^t\Pi_N\sigma_{ij}(\widetilde u^{(N)}(s))
	d\widetilde W_j^{(N)},\phi\bigg)_{L^2(\dom)}, \\
	\Lambda_i(\widetilde u,\widetilde W,\phi)(t)
	&:= (\widetilde u_i(0),\phi)_{L^2(\dom)}
	+ \sum_{j=1}^n\int_0^t\big\langle\diver\big(A_{ij}(\widetilde u(s))
	\na\widetilde u_j(s)\big),\phi\big\rangle ds \\
	&\phantom{xx}{}+ \bigg(\sum_{j=1}^n\int_0^t\sigma_{ij}(\widetilde u(s))
	d\widetilde W_j(s),\phi\bigg)_{L^2(\dom)}, 
\end{align*}
for $t\in[0,T]$ and $i=1,\ldots,n$. The following corollary is essentially a consequence
of Lemma \ref{lem.E}. 

\begin{corollary}\label{coro.E}
It holds for any $\phi_1\in L^2(\dom)$ and any $\phi_2\in H^3(\dom)$ 
satisfying $\na\phi_2\cdot\nu=0$ on $\pa\dom$ that
\begin{align*}
  \lim_{N\to\infty}\big\|(\widetilde u^{(N)},\phi_1)_{L^2(\dom)}
	- (\widetilde u,\phi_1)_{L^2(\dom)}\big\|_{L^2(\widetilde\Omega\times(0,T))} &= 0, \\
  \lim_{N\to\infty}\big\|\Lambda_i^{(N)}(\widetilde u^{(N)},\widetilde W^{(N)},\phi_2)
	- \Lambda_i(\widetilde u,\widetilde W,\phi_2)\big\|_{L^1(\widetilde\Omega\times(0,T))}
	&= 0.
\end{align*}
\end{corollary}

\begin{proof}
The first convergence follows immediately from the identity
$$
  \big\|(\widetilde u^{(N)},\phi_1)_{L^2(\dom)}
	- (\widetilde u,\phi_1)_{L^2(\dom)}\big\|_{L^2(\widetilde\Omega\times(0,T)}
	= \widetilde\E\int_0^T\big|\big(\widetilde u^{(N)}(t)-\widetilde u(t),
	\phi_1\big)_{L^2(\dom)}\big|^2 dt
$$
and convergence \eqref{3.E1}. For the second convergence, 
let $\phi_2\in H^3(\dom)$ satisfying $\na\phi_2\cdot\nu=0$ on $\pa\dom$.
Fubini's theorem implies that
\begin{align*}
  \big\|\Lambda_i^{(N)}&(\widetilde u^{(N)},\widetilde W^{(N)},\phi_2)
	- \Lambda_i(\widetilde u,\widetilde W,\phi_2)
	\big\|_{L^1(\widetilde\Omega\times(0,T))} \\
	&= \int_0^T\widetilde\E\Big|\Lambda_i^{(N)}(\widetilde u^{(N)},
	\widetilde W^{(N)},\phi_2)
	- \Lambda_i(\widetilde u,\widetilde W,\phi_2)\Big| dt.
\end{align*}
Convergences \eqref{3.E2}-\eqref{3.E4} show that each term in the definition of
$\Lambda_i^{(N)}(\widetilde u^{(N)},\widetilde W^{(N)},\phi_2)$ tends 
to the corresponding terms in 
$\Lambda_i(\widetilde u,\widetilde W,\phi_2)$ at least in
the space $L^1(\widetilde\Omega\times(0,T))$.
\end{proof}

Since $u^{(N)}$ is a strong solution to \eqref{2.approx1}-\eqref{2.approx2},
it satisfies the identity
$$
  (u_i^{(N)}(t),\phi)_{L^2(\dom)} = \Lambda_i^{(N)}(u^{(N)},W,\phi)(t)
	\quad\Prob\mbox{-a.s.}
$$
for all $t\in[0,T]$, $i=1,\ldots,n$, and $\phi\in H^1(\dom)$ and in particular, we have
$$
  \int_0^T\E\Big|(u_i^{(N)}(t),\phi)_{L^2(\dom)} 
	- \Lambda_i^{(N)}(u^{(N)},W,\phi)(t)\Big|dt = 0.
$$
Since the laws $\Law(u^{(N)},W)$ and $\Law(\widetilde u^{(N)},\widetilde W^{(N)})$
coincide, we find that
$$
  \int_0^T\widetilde\E\Big|(\widetilde u_i^{(N)}(t),\phi)_{L^2(\dom)} 
	- \Lambda_i^{(N)}(\widetilde u^{(N)},\widetilde W^{(N)},\phi)(t)\Big|dt = 0.
$$
By Corollary \ref{coro.E}, the limit $N\to\infty$ in this equation yields
$$
  \int_0^T\widetilde\E\Big|(\widetilde u_i(t),\phi)_{L^2(\dom)}
	- \Lambda_i(\widetilde u,\widetilde W,\phi)(t)\Big|dt = 0, \quad i=1,\ldots,n.
$$
This identity holds for all $\phi\in H^3(\dom)$ satisfying $\na\phi\cdot\nu=0$
on $\pa\dom$. By a density argument, it also holds for all $\phi\in H^1(\dom)$.
Hence, for Lebesgue-a.e.\ $t\in(0,T]$ and $\widetilde\Prob$-a.e.\ $\omega\in\widetilde
\Omega$, we deduce that
$$
  (\widetilde u_i(t),\phi)_{L^2(\dom)} - \Lambda_i(\widetilde u,\widetilde W,\phi)(t)
	= 0, \quad i=1,\ldots,n.
$$
By definition of $\Lambda_i$, this means that for Lebesgue-a.e.\ $t\in(0,T]$ and 
$\widetilde\Prob$-a.e.\ $\omega\in\widetilde\Omega$,
\begin{align*}
  (\widetilde u_i(t),\phi)_{L^2(\dom)}
	&= (\widetilde u(0),\phi)_{L^2(\dom)} 
	+ \sum_{j=1}^n\int_0^t\big\langle\diver\big(A_{ij}(\widetilde u(s))
	\na\widetilde u_j(s)\big),\phi\big\rangle ds \\
  &\phantom{xx}{}+ \bigg(\sum_{j=1}^n\int_0^t\sigma_{ij}(\widetilde u(s))
	d\widetilde W_j(s),\phi\bigg)_{L^2(\dom)}.
\end{align*}
Setting $\widetilde U:=(\widetilde\Omega,\widetilde\F,\widetilde\Prob,\widetilde\Fil)$,
we infer that the system $(\widetilde U,\widetilde W,\widetilde u)$ is a martingale
solution to \eqref{1.eq} and the stochastic process $\widetilde u$ satisfies
estimates \eqref{3.uH1} and \eqref{3.uL2}. 

\begin{remark}[Random initial data]\label{rem.ic}\rm
The initial data may be chosen to be random, i.e., we prescribe an inital probability 
measure $\mu^0$ on $L^2(\dom)$ instead of a given initial data. We assume that 
\begin{equation}\label{eq:5.18}
  \int_{L^2(\dom)} \|x\|_{L^2(\dom)}^p d\mu^0(x) < \infty\quad\mbox{for }
	p=\frac{24}{4-d}.
\end{equation} 
Now, in principle, we can carry out the whole analysis also in this case. 
Since for the given initial distribution $\mu^0$ and a given stochastic basis 
$(\Omega, \F, \Fil, \Prob)$, we have an $\mathcal{F}_0$-measurable random variable, 
which we will denote by $u^0$ and whose distribution is $\mu^0$. 
Because of assumption \eqref{eq:5.18}, we have $\E\|u^0\|_{L^2(\dom)}^p < \infty$ and 
consequently, the a priori estimates obtained in section \ref{sec.unif} still hold true. 
As before we can show that the set of measure $\{\Law(u^{(N)}):N\in\N\}$ is tight on 
$Z_T$ and therefore, by the Skorohod-Jakubowski theorem, we obtain a sequence of new 
random variables $(\widetilde{u}^{(N)})_{N\in\N}$ 
(and also a sequence of new Wiener processes) 
which have the same law as the old random variables $u^{(N)}$ on $Z_T$. 
In particular, $\Law(\tilde{u}^{(N)}(0)) = \Law(u^{(N)}(0))$ in $L^2(\dom)$ as well as
$\widetilde{u}^{(N)} \to \widetilde{u}$ in $C^0([0,T];L_w^2(\dom))$ 
$\widetilde\Prob$-a.s.\ and $\widetilde{u}^{(N)}(0) \to \widetilde{u}(0)$ 
in $L^2(\dom)$ weakly 
$\widetilde\Prob$-a.s. We conclude that $\Law(\tilde{u}(0)) = \Law(\tilde{u}^{(N)}(0)) 
= \Law(u^0) = \mu^0$. Thus, we have shown that the process $\widetilde{u}$ has 
the initial measure $\mu^0$ and therefore is the required martingale solution of 
\eqref{1.eq}.
\qed
\end{remark}


\subsection{Nonnegativity of the solutions}\label{sec.pos}

We show that if $u_i^0\ge 0$ in $\dom$ for $i=1,\ldots,n$ and condition
\eqref{1.sigma2} on $\sigma$ holds then $\widetilde u_i$ is nonnegative $\Prob$-a.s.
For this, we employ the technique of \cite{CPT16}. The idea is to approximate
the test function $f(z)=z^-=\max\{0,-z\}$ for $z\in\R$ and to use It\^o's formula.
We define as in \cite[Section 2.4]{CPT16} the following functions:
$$
  f_\eps(z) = \left\{\begin{array}{ll}
	-z &\quad\mbox{if }z\le -\eps, \\
	\displaystyle-3\left(\frac{z}{\eps}\right)^4z - 8\left(\frac{z}{\eps}\right)^3z
	- 6\left(\frac{z}{\eps}\right)^2z &\quad\mbox{if }-\eps\le z\le 0 \\
	0 &\quad\mbox{if }z\ge 0
  \end{array}\right.
$$
for $\eps>0$.
Then $f_\eps$ has at most linear growth, i.e.\ $|f_\eps(z)|\le C|z|$
for all $z\in\R$,
and the functions $f_\eps'$ and $\psi_\eps:=f_\eps f_\eps''+(f_\eps')^2$ are
bounded in $\R$. We set
$$
  F_\eps(v) = \int_\dom f_\eps(v(x))^2dx, \quad 
	F(v) = \int_\dom f(v(x))^2 dx
$$
for square-integrable functions $v:\dom\to\R$. 

We replace the diffusion coefficients
$A_{ij}(u^{(N)})$ in \eqref{2.approx1} by the modified coefficients
$$
  A_{ij}^+(u^{(N)}) = \delta_{ij}\bigg(a_{i0} + \sum_{k=1}^n a_{ik}u_k^2\bigg)
	+ 2a_{ij}u_i^+u_j, \quad i,j=1,\ldots,n,
$$
where $z^+=\min\{0,z\}$ is the positive part of $z\in\R$. 
Observe that generally, $A_{ij}^+(u)\neq A_{ij}(u)$ but if $u_i\ge 0$ for
all $i=1,\ldots,n$ then we obtain the original coefficients, $A_{ij}^+(u)=A_{ij}(u)$.
The proof of Lemma \ref{lem.ex} provides the existence of a
pathwise unique strong solution $u^{(N)}$ to this truncated problem.
The It\^o formula in finite dimensions gives \cite[Formula (3.3)]{CPT16}
\begin{align}
  F_\eps(u_i^{(N)}(t)) &= F_\eps(u_i^{(N)}(0)) \nonumber \\
	&\phantom{xx}{}+ 2\int_0^t\int_\dom f_\eps(u_i^{(N)}(s))f_\eps'(u_i^{(N)}(s))\Pi_N
	\bigg(\sum_{j=1}^n\sigma_{ij}(u^{(N)}(s))\bigg)dxdW_j(s) \nonumber \\
	&\phantom{xx}{}- 2\int_0^t\int_\dom\psi_\eps(u_i^{(N)}(s))
	\sum_{j=1}^n A_{ij}^+(u^{(N)}(s))\na u_i^{(N)}(s)\cdot\na u_j^{(N)}(s)dxds 
	\label{3.Feps} \\
	&\phantom{xx}{}+ \int_0^t\int_\dom\sum_{j=1}^n\sum_{k,\ell=1}^N\sum_{m=1}^\infty
	\psi_\eps(u_i^{(N)}(s))e_ke_\ell\sigma_{ij}^{mk}(u^{(N)}(s))
	\sigma_{ij}^{m\ell}(u^{(N)}(s))dxds \nonumber \\
	&=: I_{\eps,0}^{(N)} + I_{\eps,1}^{(N)} + I_{\eps,2}^{(N)} + I_{\eps,3}^{(N)},
	\nonumber
\end{align}
where $\sigma_{ij}^{km}$ is defined in \eqref{1.sigmadW}. We claim that the
integral $I_{\eps,1}^{(N)}$ is nonpositive. Indeed, we write
\begin{align*}
  I_{\eps,1}^{(N)} &= -2\int_0^t\int_\dom\psi_\eps(u_i^{(N)})A_{ii}^+(u^{(N)})
	|\na u_i^{(N)}|^2 dxds \\
	&\phantom{xx}{}- 2\int_0^t\int_\dom\psi_\eps(u_i^{(N)})\sum_{j\neq i}A_{ij}^+(u^{(N)})
	\na u_i^{(N)}\cdot\na u_j^{(N)} dxds.
\end{align*}
The first term on the right-hand side is clearly nonpositive; the second term vanishes
since $\psi_\eps(u_i^{(N)})=0$ in $\{u_i^{(N)}\ge 0\}$ and $A_{ij}^+(u^{(N)})=0$
in $\{u_i^{(N)}\le 0\}$. This shows that $I_{1,\eps}^{(N)}\le 0$. By \eqref{3.conv},
we know that $u^{(N)}\to u$ in $L^2(0,T;L^2(\dom))$ as $N\to\infty$.
(To be precise, we should work with the new processes $\widetilde u^{(N)}$ but
we omit the tilde.) 
Therefore, up to a subsequence which is not relabeled, $u^{(N)}\to u$
for a.e.\ $(x,t,\omega)\in\dom\times(0,T)\times\Omega$. Following the steps of
\cite[Section 3.2]{CPT16}, we can show the following $\Prob$-a.s.\ convergence 
results as $N\to\infty$:
\begin{align*}
  & F_\eps(u_i^{(N)}(t)) \to F_\eps(u_i(t)), \quad
	I_{\eps,0}^{(N)} \to F_\eps(u_i^0), \\
	& I_{\eps,2}^{(N)} \to 2\int_0^t\int_\dom f_\eps(u_i(s))f_\eps'(u_i(s))
	\sum_{j=1}^n\sigma_{ij}(u(s))dxdW_j(s), \\
	& I_{\eps,3}^{(N)} \to \int_0^t\int_\dom\sum_{j=1}^n\sum_{k,\ell=1}^\infty
	\sum_{m=1}^\infty
	\psi_\eps(u_i(s))e_ke_\ell\sigma_{ij}^{mk}(u(s))\sigma_{ij}^{m\ell}(u(s))dxds.
\end{align*}
Passing to the limit $N\to\infty$ in \eqref{3.Feps} then leads to
\begin{align*}
  F_\eps(u_i(t)) &\le F_\eps(u_i^0) + 2\int_0^t\int_\dom f_\eps(u_i(s))f'_\eps(u_i(s))
	\sum_{j=1}^n\sigma_{ij}(u(s))dW_j(s)dx \\
	&\phantom{xx}{}+\int_0^t\int_\dom
	\psi_\eps(u_i(s))\sum_{j=1}^\infty\sum_{m=1}^\infty\big(\sigma_{ij}(u(s))\eta_m\big)^2
	dxds.
\end{align*}
Taking the mathematical expectation, the stochastic integral vanishes:
\begin{equation}\label{3.Feps2}
  \E F_\eps(u_i(t)) \le \E F_\eps(u_i^0)
	+ \E\int_0^t\int_\dom\psi_\eps(u_i(s))
	\sum_{j=1}^n\sum_{m=1}^\infty\bigg(\sigma_{ij}(u(s))\eta_m\bigg)^2 dxds.
\end{equation}
It is shown in \cite[Section 3.4]{CPT16} that in the limit $\eps\to 0$,
$\Prob$-a.s.\footnote{Observe that there is a typo in 
\cite[formulas (3.21)-(3.24)]{CPT16}: 
The sum from $l=1$ to $\infty$ should be outside the brackets.}
\begin{align*}
  & \E F_\eps(u_i(t)) \to \E\|u_i^-(t)\|_{L^2(\dom)}^2, \quad
	\E F_\eps(u_i^0) \to \E\|(u_i^0)^-\|_{L^2(\dom)}^2, \\
	& \E\int_0^t\int_\dom\psi_\eps(u_i)\sum_{j=1}^n\sum_{m=1}^\infty\bigg(\sigma_{ij}(u)
	\eta_m\bigg)^2 dxds \to \E\int_0^t\sum_{j=1}^n
	\|\sigma_{ij}(-u^-)\|_{\T_2(Y;L^2(\dom))}^2 ds.
\end{align*}
Thus, the limit $\eps\to 0$ in \eqref{3.Feps2} gives
$$
  \E\|u_i^-(t)\|_{L^2(\dom)}^2 \le \E\|(u_i^0)^-\|_{L^2(\dom)}^2
	+ \E\int_0^t\sum_{j=1}^n\|\sigma_{ij}(-u_i^-(s))\|_{\T_2(Y;L^2(\dom))}^2 ds.
$$
The first term on the right-hand side vanishes since $u_i^0\ge 0$. For the
second term, we employ the linear growth \eqref{1.sigma2} of $\sigma_{ij}$,
showing that
$$
  \E\|u_i^-(t)\|_{L^2(\dom)}^2 
	\le \E\int_0^t\|u_i^-(s)\|_{L^2(\dom)}^2 ds.
$$
Gronwall's lemma implies that $\E\|u_i^-(t)\|_{L^2(\dom)}^2=0$ for $t\in(0,T)$
and consequently, $u_i(t)\ge 0$ in $\dom$, $\Prob$-a.s. for a.e.\ $t\in[0,T]$
and all $i=1,\ldots,n$. This finishes the proof.


\begin{appendix}
\section{Some results from stochastic analysis}\label{app}

\subsection{Results for stochastic processes}

The following particular It\^o formula is proved in \cite[Theorem 4.2.5]{PrRo07}.

\begin{theorem}[It\^o formula]\label{thm.ito}
Let $V\subset H\subset V'$ be a Gelfand triple and $U$ be a separable Hilbert space, 
$X_0\in L^2(\Omega;H)$, and let $a\in L^2(\Omega\times(0,T);V')$, 
$b\in L^2(\Omega\times(0,T);\T_2(U,H))$ be progressively measurable. Define the 
stochastic process
$$
  X(t) = X_0 + \int_0^t a(s)ds + \int_0^t b(s)dW(s), \quad t\in(0,T).
$$
Then
\begin{align*}
  \frac12\|X(t)\|_H^2 
	&= \frac12\|X_0\|_H^2 + \int_0^t \langle a(s),X(s)\rangle_{V',V} ds
	+ \frac12\int_0^t\|b(s)\|_{\T_2(U,H)}^2 ds \\
	&\phantom{xx}{}+ \int_0^t(X(s),b(s)dW(s))_H \quad\mbox{for }t\in(0,T),
\end{align*}
where $\langle\cdot,\cdot\rangle_{V',V}$ is the duality pairing between $V'$
and $V$, $(\cdot,\cdot)_H$ is the inner product in $H$, and $X(s)\in
L^2(\Omega\times(0,T);V)$ in $\langle a(s),X(s)\rangle_{V',V}$
is any $V$-valued progressively measurable 
$dt\otimes\Prob$ version of the equivalence class represented by $X(s)$.
\end{theorem}

The next proposition can be found in \cite[Prop. 2.10]{Kru14}.

\begin{proposition}[It\^o isometry]\label{prop.iso}
Let $\sigma(u)\in L^2((0,T)\times\Omega;\T_2(Y;L^2(\dom)))$ be a predictable
stochastic process. Then
$$
  \E\bigg(\int_0^T\sigma(u(s))dW(s)\bigg)^2 
	= \E\int_0^T\|\sigma(u)\|_{\T_2(Y;L^2(\dom))}^2ds.
$$
\end{proposition}

This result can be generalized in the following sense; see \cite[Prop. 2.12]{Kru14}.

\begin{proposition}[Burkholder-Davis-Gundy inequality]\label{prop.bdg}
Let $p\ge 2$ and let 
$\sigma:L^2(\dom)\times[0,T]\times\Omega\to \T_2(Y;L^2(\dom))$ be
a predictible stochastic process such that
$$
  \E\bigg(\int_0^T\|\sigma(u(s))\|_{\T_2(Y;L^2(\dom))}^2 ds\bigg)^{p/2} < \infty.
$$
Then, for some $C>0$ depending on $p$,
$$
  \E\bigg|\int_0^T\sigma(u(s))dW(s)\bigg|^p
	\le C\E\bigg(
	\int_0^T\|\sigma(u(s))\|_{\T_2(Y;L^2(\dom))}^2 ds\bigg)^{p/2}.
$$
\end{proposition}

\subsection{Finite-dimensional stochastic differential equations}

We state a result on the 
existence of the pathwise unique strong solution to the stochastic differential
equation on $\R^n$ (essentially taken from \cite[Theorem 3.1.1]{PrRo07}; originally 
from \cite{Kry99}),
\begin{equation}\label{2.sde}
  \pi\cdot dX(t) = a(X,t)dt + b(X,t)dW(t), \quad t>0, \quad X(0)=X_0.
\end{equation}
Here, $\pi=(\pi_1,\ldots,\pi_n)\in(0,\infty)^n$, 
$a:\R^n\times[0,\infty)\times\Omega\to\R^n$ and
$b:\R^n\times[0,\infty)\times\Omega\to\R^{n\times m}$ are both continuous
in $x\in\R^n$ for fixed $t\in[0,\infty)$, $\omega\in\Omega$, 
progressively measurable, and satisfy
for all $R$, $T>0$,
\begin{equation}\label{2.ab1}
  \int_0^T\sup_{|x|\le R}\big(|a(x,t)|^2 + |b(x,t)|^2\big)dt < \infty
	\quad\mbox{in }\Omega,
\end{equation}
where $|a(x,t)|$ is the Euclidean norm on $\R^n$ and $|b(x,t)|$ is the 
Frobenius norm on $\R^{n\times m}$. Furthermore, we assume that for all
$R$, $t>0$, and $x$, $y\in\R^n$ with $|x|$, $|y|\le R$,
\begin{equation}\label{2.ab2}
\begin{aligned}
  2\big(a(x,t)-a(y,t),x-y\big) + \big|b(x,t)-b(y,t)\big|^2 &\le K_R(t)|x-y|^2, \\
	2(a(x,t),x) + |b(x,t)|^2 &\le K_1(t)(1+|x|^2),
\end{aligned}
\end{equation}
where for every $R>0$, $K_R(t)$ is an $\R_+$-valued $\F_t$-adapted process satisyfing
$\int_0^T K_R(t)dt<\infty$ in $\Omega$ for all $R$, $T>0$. 
We call $X$ the pathwise strong solution to \eqref{2.sde} if 
$X(t)=(X_1(t),\ldots,X_n(t))$ 
for $t\ge 0$ is a $\Prob$-a.s.\ continuous $\R^n$-valued $\F_t$-adapted
process such that $\Prob$-a.s. for all $t\ge 0$,
\begin{equation}\label{2sde2}
  \pi_i X_i(t) = \pi_i X_{0i} + \int_0^t a_i(X(s),s)ds 
	+ \int_0^t \sum_{j=1}^m b_{ij}(X(s),s)dW_j(s), \quad
	i=1,\ldots,n.
\end{equation} 

\begin{theorem}[Existence of solutions]\label{thm.sde}
Let Assumptions \eqref{2.ab1}-\eqref{2.ab2} hold and let $X_0:\Omega\to\R^n$
be $\F_0$-measurable. Then there exists a (up to $\Prob$-indistinguishability)
pathwise unique strong solution to \eqref{2.sde}.
\end{theorem}

The proof is the same as in \cite[Theorem 3.1.1]{PrRo07}. The difference to this
theorem is the appearance of the constant vector $\pi$ 
on the left-hand side of \eqref{2.sde}.
As the proof in \cite{PrRo07} is based on the Euler method and the vector is constant, 
this appearance does
not change the arguments. We just have to take into account that
$\min_{i=1,\ldots,n}\pi_i$ is positive.

\subsection{Tightness}

We recall some definitions and results on the tightness of
families of probability measures.
Let $E$ be a separable Banach space with norm $\|\cdot\|_E$ and associated Borel
$\sigma$-field $\B(E)$. 

\begin{definition}[Tightness]
The family $\Lambda$ of probability measures on $(E,\B(E))$ is said to be
{\em tight} if and only if for any $\eps>0$, there exists a compact set $K_\eps\subset E$
such that 
$$
  \mu(K_\eps)\ge 1-\eps\quad\mbox{for all }\mu\in\Lambda.
$$
\end{definition}

The theorem of Skorokhod allows for the representation of the limit measure of a
weakly convergent sequence of probability measures on a metric space
as the law of a pointwise convergent
sequence of random variables defined on a common probability space. 
Since our space $Z_T$, defined in \eqref{3.Z}, is not a metric space,
we use Jakubowski's generalization of the Skorokhod Theorem,
in the form given in \cite[Theorem C.1]{BrOn10} (see the original theorem in
\cite{Jak97}). This version is valid for topological spaces. 

\begin{theorem}[Skorokhod-Jakubowski]\label{thm.skoro}
Let $Z$ be a topological space such that there exists a sequence $(f_m)_{m\in\N}$
of continuous functions $f_m:Z\to\R$ that separate points of $Z$. Let $S$ be the
$\sigma$-algebra generated by $(f_m)_{m\in\N}$. Then
\begin{enumerate}
\item Every compact subset of $Z$ is metrizable.
\item If $(\mu_m)_{m\in\N}$ is a tight sequence of probability measures on $(Z,S)$,
then there exists a subquence $(\mu_{m_k})_{k\in\N}$, a probability space
$(\widetilde\Omega,\widetilde\F,\widetilde\Prob)$, and $Z$-valued Borel measurable
random variables $\xi_k$ and $\xi$ such that {\rm (i)} $\mu_{m_k}$ is the law of $\xi_k$
and {\em (ii)} $\xi_k\to\xi$ almost surely on $\Omega$.
\end{enumerate}
\end{theorem}

The Aldous condition is mentioned in the tightness criterion of Theorem \ref{thm.tight},
and therefore we recall its definition.

\begin{definition}[Aldous condition]\label{def.ald}
Let $(X_n)_{n\in\N}$ be a sequence of stochastic processes on a complete
separable metric space $S$, defined on the probability space $(\Omega,\F,\Prob)$
with filtration $\Fil=(\F_{t})_{t\in[0,T]}$.
We say that $(X_n)_{n\in\N}$ satisfies the {\em Aldous condition} 
if and only if for any $\eps>0$, there exists $\eta>0$ 
such that for any $\delta>0$ and any sequence $(\tau_n)_{n\in\N}$
of $\Fil$-stopping times with $\tau_n\le T$, it holds that
$$
  \sup_{n\in\N}\sup_{0<\theta<\delta}\Prob\big\{d\big(X_n(\tau_n+\theta),
	X_n(\tau_n)\big)\ge\eta\big\} \le \eps.
$$
\end{definition}

\subsection{Vitali's convergence theorem}

We use the following version of Vitali's convergence theorem (which can be seen
as a special version of the theorem of De la Vall\'ee-Poussin).

\begin{theorem}[Vitali]
Let $(a_N)$ be a sequence of integrable functions on some probability space 
$(\Omega,\B(\Omega),\Prob)$ such that $a_N\to a$ a.e.\ as $N\to\infty$
(or $a_N\to a$ in measure) for some integrable function $a$ and there exist $r>1$ and
a constant $C>$ such that $\E|a_N|^r\le C$ for all $N\in\N$. Then
$\E|a_N|\to\E|a|$ as $N\to\infty$.
\end{theorem}

\end{appendix}


\end{document}